\documentclass[reqno, twoside, a4paper]{amsart}

\usepackage{amsmath}
\usepackage{amssymb}
\usepackage{amsthm}
\usepackage{graphicx}
\usepackage{subfig}
\usepackage{enumerate}
\usepackage{hyperref}

\hypersetup{%
pdftitle={Symplectic fillings of quotient surface singularities and minimal model program},%
pdfauthor={Hakho Choi, Heesang Park, Dongsoo Shin},%
pdfkeywords={minimal model program, quotient surface singularity, rational blow-down, symplectic antiflip, symplectic filling},%
colorlinks=true,%
citecolor=black,%
filecolor=black,%
linkcolor=black,%
urlcolor=black,%
}

\usepackage{comment}
\usepackage{xspace}
\usepackage{mathtools}

\usepackage{mathptmx}       
\usepackage{helvet}         
\usepackage{courier}        

\usepackage{tikz}

\usetikzlibrary{patterns,decorations.pathreplacing}
\usetikzlibrary{decorations.pathmorphing}

\tikzset{bullet/.style={
shape = circle,fill = black, inner sep = 0pt, outer sep = 0pt, minimum size = 0.35em, line width = 0pt, draw=black!100}}

\tikzset{circle/.style={
shape = circle,fill = none, inner sep = 0pt, outer sep = 0pt, minimum size = 0.35em, line width = 1pt, draw=black!100}}

\tikzset{rectangle/.style={
shape = rectangle,fill = white, inner sep = 0pt, outer sep = 0pt, minimum size = 0.35em, line width = 0pt, draw=black!100}}

\tikzset{empty/.style={
shape = circle,fill = white, inner sep = 0pt, outer sep = 0pt, minimum size = 0.35em, line width = 0pt, draw=white!100}}

\tikzset{xmark/.style={
shape = x,fill = white, inner sep = 0pt, outer sep = 0pt, minimum size = 0em, line width = 0pt, draw=white!100}}

\tikzset{longrectangle/.style={
inner sep = 1em,
rectangle,
minimum size=1em,
very thick,
draw=black!100, 
}}

\tikzset{label distance=-0.15em}

\tikzset{font=\scriptsize}

\usepackage{tikz-cd}
\tikzcdset{every label/.append style = {font = \normalsize}}
\usepackage{tkz-graph}

\newtheorem{theorem}{Theorem}[section]
\newtheorem{lemma}[theorem]{Lemma}
\newtheorem{proposition}[theorem]{Proposition}
\newtheorem{corollary}[theorem]{Corollary}

\theoremstyle{definition}

\newtheorem{definition}[theorem]{Definition}
\newtheorem{remark}[theorem]{Remark}

\numberwithin{equation}{section}

\begin{document}

\title[Symplectic fillings and minimal model program]{Symplectic fillings of quotient surface singularities and minimal model program}

\author[H. Choi]{Hakho Choi}

\address{School of Mathematics, Korea Institute for Advanced Study, 85 Hoegiro, Dongdaemun-gu, Seoul 02455, Republic of Korea}

\email{hakho@kias.re.kr}

\author[H. Park]{Heesang Park}

\address{Department of Mathematics, Konkuk University, Seoul 05029 \& School of Mathematics, Korea Institute for Advanced Study, 85 Hoegiro, Dongdaemun-gu, Seoul 02455, Republic of Korea}

\email{HeesangPark@konkuk.ac.kr}

\author[D. Shin]{Dongsoo Shin}

\address{Department of Mathematics, Chungnam National University, Daejeon 34134, Republic of Korea}

\email{dsshin@cnu.ac.kr}



\begin{abstract}
We prove that every minimal symplectic filling of the link of a quotient surface singularity can be obtained from its minimal resolution by applying a sequence of rational blow-downs and symplectic antiflips. We present an explicit algorithm inspired by the minimal model program for complex 3-dimensional algebraic varieties.
\end{abstract}

\maketitle

\section{Introduction}

Let $(X,0)$ be a quotient surface singularity. A \emph{symplectic filling} of $X$ is a symplectic filling of the link $L$ of $X$, that is, a symplectic $4$-manifold $W$ with $L$ as its boundary such that the induced contact structure on $L$ coming from the symplectic structure of the interior of $W$ is compatible with the Milnor fillable contact structure on $L$.

Lisca~\cite{Lisca-2008} classify symplectic fillings of cyclic quotient surface singularities (up to symplectic deformation equivalence) as complements of certain symplectic spheres (depending only on the singularities) in certain rational symplectic $4$-manifolds. Bhupal-Ono~\cite{Bhupal-Ono-2012} (refer Bhupal-Ono~\cite{Bhupal-Ono-2017} also) classify that of non-cyclic quotient surface singularities also as the complements. Meanwhile, Bhupal-Ozbagci~\cite{Bhupal-Ozbagci-2016} show that every minimal symplectic filling of $X$ can be constructed from the minimal resolution of $X$ by a sequence of rational blow-downs in some sense. For this, they construct a positive allowable Lefschetz fibration over the disk on each minimal symplectic fillings of any cyclic quotient surface singularity. Using a similar strategy, Choi-Park~\cite{Choi-Park-2018} prove a similar result for non-cyclic quotient surface singularities.

On the other hand, PPSU~\cite{PPSU-2018} prove a comparable result to Bhupal-Ozbagci~\cite{Bhupal-Ozbagci-2016} and Choi-Park~\cite{Choi-Park-2018} by using techniques from algebraic geometry. PPSU~\cite[Theorem~11.3]{PPSU-2018} show that minimal symplectic fillings of quotient surface singularities are exactly their Milnor fibers: For any minimal symplectic filling $W$ of a quotient surface singularity $X$ there is a smoothing $\pi \colon (X \subset \mathcal{X}) \to (0 \in \mathbb{D})$ over a small disk $\mathbb{D} (\subset \mathbb{C})$ such that the Milnor fiber of the smoothing $\pi$ (i.e., a general fiber $X_t = \pi^{-1}(t)$ ($t \neq 0)$) is diffeomorphic to the symplectic filling $W$. By the way, KSB~\cite[Theorem~3.9]{Kollar-Shepherd-Barron-1988} show that every Milnor fiber of a quotient surface singularity $X$ is a general fiber of a $\mathbb{Q}$-Gorenstein smoothing of a certain special partial resolution (called, \emph{$P$-resolution}) of $X$ and that every $P$-resolution is dominated by the so-called maximal resolution of $X$, where a $\mathbb{Q}$-Gorenstein smoothing may be regarded as an analogue of the rational blow-down surgery. Putting it all together in the language of topology, PPSU~\cite[Theorem~11.3]{PPSU-2018} show that every minimal symplectic filling of a quotient surface singularity is obtained from its maximal resolution by a sequence of rational blow-downs. But maximal resolutions are not equal to minimal resolutions in general. They are obtained from minimal resolutions by ordinary blow-ups; cf.~KSB~\cite[Lemma~3.13]{Kollar-Shepherd-Barron-1988}.

In this paper we produce an explicit algorithm for constructing Milnor fibers (that is, $\mathbb{Q}$-Gorenstein smoothings of $P$-resolutions) symplectically from the minimal resolutions (not from maximal resolutions), where the algorithm is inspired by the minimal model program of 3-dimensional complex varieties. As a result, we show that:

\begin{theorem}\label{theorem:main}
For each minimal symplectic filling $W$ of a quotient surface singularity $X$, there is a sequence of rational blow-downs and symplectic antiflips that transforms the minimal resolution of $\widetilde{X}$ to a $4$-manifold that is diffeomorphic to the symplectic filling $W$.
\end{theorem}

Here a \emph{symplectic antiflip} is a converse operation to the so-called \emph{symplectic flip} which is a way to convert a rationally blown-down regular neighborhood into a simpler one; See Section~\ref{section:symplectic-flips}.

The difference with Bhupal-Ozbagci~\cite{Bhupal-Ozbagci-2016} or Choi-Park~\cite{Choi-Park-2018} is that it is unavoidable to use symplectic antiflips in our theorem in most cases. Indeed symplectic antiflips tell us clearly where to rationally blow down to get the given symplectic fillings. But the authors believe that a similar operation to symplectic antiflip would be implemented in their results too. We discuss this in Section~\ref{subsection:Examples} with some examples.

\subsection*{Notations and notions}

In the dual graph of a bunch of $\mathbb{CP}^1$'s, we decorate by rectangles $\square$ those curves that are contracted to a singular point. For example, the dual graph
\begin{equation}\label{equation:notation-square}
\begin{tikzpicture}[scale=1]
\node[rectangle] (00) at (0,0) [label=above:{$-2$}] {};
\node[rectangle] (10) at (1,0) [label=above:{$-4$}] {};
\node[rectangle] (20) at (2,0) [label=above:{$-3$}] {};
\node[rectangle] (30) at (3,0) [label=above:{$-3$}] {};
\node[bullet] (40) at (4,0) [label=above:{$-2$},label=below:{$C$}] {};

\draw[-] (00)--(10);
\draw[-] (10)--(20);
\draw[-] (20)--(30);
\draw[-] (30)--(40);
\end{tikzpicture}
\end{equation}
implies that we contract the linear chain of $\mathbb{CP}^1$'s whose dual graph is
\begin{equation*}
\begin{tikzpicture}[scale=1]
\node[bullet] (00) at (0,0) [label=above:{$-2$}] {};
\node[bullet] (10) at (1,0) [label=above:{$-4$}] {};
\node[bullet] (20) at (2,0) [label=above:{$-3$}] {};
\node[bullet] (30) at (3,0) [label=above:{$-3$}] {};

\draw[-] (00)--(10);
\draw[-] (10)--(20);
\draw[-] (20)--(30);
\end{tikzpicture}
\end{equation*}
so that we obtain a singular surface consisting of the $-2$-curve $C$ with a cyclic quotient surface singularity $\frac{1}{50}(1,29)$ on $C$.

We denote for simplicity by \[a_1-a_2-\dotsb-a_n\] a linear chain of $\mathbb{CP}^1$'s (or $2$-spheres) whose dual graph is given as
\begin{equation*}
\begin{tikzpicture}[scale=1]
\node[bullet] (10) at (1,0) [label=above:{$-a_1$}] {};
\node[bullet] (20) at (2,0) [label=above:{$-a_2$}] {};

\node[empty] (250) at (2.5,0) [] {};
\node[empty] (30) at (3,0) [] {};

\node[bullet] (350) at (3.5,0) [label=above:{$-a_n$}] {};

\draw [-] (10)--(20);
\draw [-] (20)--(250);
\draw [dotted] (20)--(350);
\draw [-] (30)--(350);
\end{tikzpicture}
\end{equation*}
with $a_i > 0$. We then enclose the contracted curves by the brackets $[ \quad ]$ instead of rectangles $\square$. For example, $[2, 4, 3, 3]-2$ is equal to the dual graph in Equation~\eqref{equation:notation-square}. Sometimes $[a_1,\dotsc,a_n]$ denote the singularity itself that is obtained by contracting the linear chain $a_1-\dotsb-a_n$.

We also regard dual graphs of symplectic $2$-spheres as regular neighborhoods of the spheres given by plumbing construction. In this case, $2$-spheres decorated by $\square$ or enclosed by brackets $[ \quad ]$ are those $2$-spheres that are rationally blown down. So the above example
\begin{tikzpicture}[scale=0.5]
\node[rectangle] (00) at (0,0) [label=above:{$-2$}] {};
\node[rectangle] (10) at (1,0) [label=above:{$-4$}] {};
\node[rectangle] (20) at (2,0) [label=above:{$-3$}] {};
\node[rectangle] (30) at (3,0) [label=above:{$-3$}] {};
\node[bullet] (40) at (4,0) [label=above:{$-2$}] {};

\draw[-] (00)--(10);
\draw[-] (10)--(20);
\draw[-] (20)--(30);
\draw[-] (30)--(40);
\end{tikzpicture}
or $[2,4,3,3]-2$ denote a regular neighborhood of
\begin{tikzpicture}[scale=0.5]
\node[bullet] (00) at (0,0) [label=above:{$-2$}] {};
\node[bullet] (10) at (1,0) [label=above:{$-4$}] {};
\node[bullet] (20) at (2,0) [label=above:{$-3$}] {};
\node[bullet] (30) at (3,0) [label=above:{$-3$}] {};
\node[bullet] (40) at (4,0) [label=above:{$-2$}] {};

\draw[-] (00)--(10);
\draw[-] (10)--(20);
\draw[-] (20)--(30);
\draw[-] (30)--(40);
\end{tikzpicture}
that is rationally blown down along
\begin{tikzpicture}[scale=0.5]
\node[bullet] (00) at (0,0) [label=above:{$-2$}] {};
\node[bullet] (10) at (1,0) [label=above:{$-4$}] {};
\node[bullet] (20) at (2,0) [label=above:{$-3$}] {};
\node[bullet] (30) at (3,0) [label=above:{$-3$}] {};

\draw[-] (00)--(10);
\draw[-] (10)--(20);
\draw[-] (20)--(30);
\end{tikzpicture}

\subsection*{Acknowledgements}

Heesang Park was supported by Basic Science Research Program through the National Research Foundation of Korea(NRF) funded by the Ministry of Education: NRF-2018R1D1A1B07042134. This paper was written as part of Konkuk University's research support program for its faculty on sabbatical leave in 2019. Dongsoo Shin was supported by Basic Science Research Program through the National Research Foundation of Korea(NRF) funded by the Ministry of Education: NRF-2018R1D1A1B07048385. A part of this work was done when HP and DS were on sabbatical leave at Korea Institute for Advanced Study. They would like to thank KIAS for warm hospitality and financial support.

\section{Symplectic fillings and $P$-resolutions}
\label{section:symplectic-filling-P-resolution}

We recall basics on the correspondence between minimal symplectic fillings and $P$-resolutions of quotient surface singularities.

\subsection{Singularities of class $T$}

We first introduce cyclic quotient surface singularities that admit smoothings whose Milnor fibers are rational homology disk. For details, refer KSB~\cite[\S3]{Kollar-Shepherd-Barron-1988} for example.

Let $(X,0)$ be a normal surface singularity. A \emph{smoothing} $\pi \colon \mathcal{X} \to \mathbb{D}$ over a small disk $0 \in \mathbb{D} (\subset \mathbb{C})$ is a surjective flat morphism such that $\pi^{-1}(0) \cong X$ and a general fiber $X_t := \pi^{-1}(t)$ ($t \neq 0$) is smooth. All general fibers are diffeomphic to each other. So we call a general fiber of a smoothing $\pi \colon (X \subset \mathcal{X}) \to (0 \in \mathbb{D})$ of $X$ the \emph{Milnor fiber} of the smoothing $\pi$. A smoothing $\mathcal{X} \to \Delta$ of $X$ is \emph{$\mathbb{Q}$-Gorenstein} if $K_{\mathcal{X}}$ is $\mathbb{Q}$-Cartier.

\begin{definition}
A normal surface singularity is \emph{of class $T$} if it is a quotient surface singularity, and it admits a $\mathbb{Q}$-Gorenstein one-parameter smoothing.
\end{definition}

\begin{proposition}[{KSB~\cite[Proposition~3.10]{Kollar-Shepherd-Barron-1988}}]
A singularity of class $T$ is a rational double point or a cyclic quotient surface singularity of type $\frac{1}{dn^2}(1, dna-1)$ with $d \ge 1$, $n \ge 2$, $1 \le a < n$, and $(n,a)=1$.
\end{proposition}

A \emph{singularity of class $T_0$} is defined by a cyclic quotient surface singularity of type $\frac{1}{n^2}(1, na-1)$ with $n > a \ge 1$ and $(n,a)=1$. Any singularity of class $T_0$ admits a smoothing whose Milnor fiber $M$ is a rational homology disk, that is, $H^i(M;\mathbb{Q})=0$ for $i \ge 1$. Furthermore, according to Wahl~\cite[Example~5.9.1]{Wahl-1981} and Looijenga–Wahl~\cite[Remark~5.10]{Looijenga-Wahl-1986}, singularities of class $T_0$ are the only cyclic quotient singularities having a rational homology disk smoothing.

So the one-parameter $\mathbb{Q}$-Gorenstein smoothing of a singularity of class $T_0$ may be interpreted topologically as a rational blow-down surgery along its minimal resolution defined in Fintushel--Stern~\cite{Fintushel-Stern-1997} and J.~Park~\cite{JPark-1997}, which is called as the \emph{rational blow-down along a singularity of class $T_0$}. Notice that a rational blow-down surgery along a singularity of class $T_0$ is a symplectic surgery by Symington~\cite{Symington-1988}, \cite{Symington-2001}. See also Park-Stipsicz~\cite{Park-Stipsicz 2014}.

Any singularities of class $T_0$ can be obtained by the following iterations (KSB~\cite[Proposition~3.11]{Kollar-Shepherd-Barron-1988}): At first, $[4]$ is a singularity of class $T_0$. If the singularity $[b_1, \dotsc, b_r]$ is of class $T_0$, then so are $[2, b_1, \dotsc, b_{r-1}, b_r+1]$ and $[b_1+1, b_2, \dotsc, b_r, 2]$. Then every singularity of class $T_0$ can be obtained by starting with $[4]$ and iterating the above steps.

\begin{definition}
Let $[e_1,\dotsc,e_n]$ be a singularity of class $T_0$ and let $e_i$ be the image of $[4]$ under the above procedure. The corresponding exceptional curve $E_i$ is called the \emph{initial curve} of the singularity.
\end{definition}

For example, in $[6,2,2]$, $E_1$ is the initial curve. On the other hand, in $[2,5,3]$, $E_2$ is the initial one.

\subsection{$P$-resolutions and $M$-resolutions}

Let $(X,0)$ be a quotient surface singularity. There are one-to-one correspondence between Milnor fibers and certain partial resolutions of $X$.

\begin{definition}[{KSB~\cite[Definition~3.8]{Kollar-Shepherd-Barron-1988}}]
A partial resolution $f \colon Y \to X$ is called a \emph{$P$-resolution} of $X$ if $Y$ has only singularities of class $T$, and $K_Y$ is ample relative to $f$.
\end{definition}

\begin{remark}
Every $P$-resolution of $X$ is dominated by the so-called maximal resolution of $X$, which can be obtained by blowing up the minimal resolution of $X$; KSB~\cite[Lemma~3.14]{Kollar-Shepherd-Barron-1988}.
\end{remark}

According to KSB~\cite[Theorem~3.9]{Kollar-Shepherd-Barron-1988}, each smoothing $\pi \colon \mathcal{X} \to \mathbb{D}$ of $X$ is induced by a $\mathbb{Q}$-Gorenstein smoothing $\phi \colon \mathcal{Y} \to \mathbb{D}$ of a $P$-resolution $Y$ of $X$. That is, there is a birational morphism $\mathcal{Y} \to \mathcal{X}$ over $\mathbb{D}$ such that the Milnor fiber of the smoothing $\pi$ of $X$ is isomorphic to a general fiber $Y_t = \phi^{-1}(t)$ ($t \neq 0$) of the smoothing $\phi$ of $Y$.

Furthermore one may restrict the types of singularities on $P$-resolutions.

\begin{definition}[{Behnke--Christophersen~\cite[p.882]{Behnke-Christophersen-1994}}]
\label{definition:$M$-resolution}
An \emph{$M$-resolution} of a quotient surface singularity $(X,0)$ is a partial resolution $f \colon Z \to X$ such that $Z$ has only singularities of class $T_0$ as its singularities and $K_Z$ is nef relative to $f$, i.e., $K_Z \cdot E \ge 0$ for all $f$-exceptional curves $E$.
\end{definition}

\begin{remark}
There is an easy way to convert a given $P$-resolution of a quotient surface singularity $X$ to the corresponding $M$-resolution by blowing up the minimal resolution of the $P$-resolution and contracting certain parts of it; cf. PPSU~\cite[\S6]{PPSU-2018}. Hence the minimal resolution of a $M$-resolution can be obtained by blowing up the minimal resolution of $X$. For example, for the $P$-resolution $[2,4,3,3]-2$, its corresponding $M$-resolution is given by $[2,5,3]-1-[2,5,3]-2$.
\end{remark}

Behnke--Christophersen~\cite[3.1.4, 3.3.2, 3.4]{Behnke-Christophersen-1994} also prove a similar result to the above KSB~\cite[Theorem~3.9]{Kollar-Shepherd-Barron-1988}: Every Milnor fiber of a smoothing of a quotient surface singularity $X$ is isomorphic to a general fiber of the $\mathbb{Q}$-Gorenstein smoothing of the corresponding $M$-resolution of $X$.

Since the Milnor fiber of a rational double point is symplectomorphic to its minimal resolution, the above result of Behnke--Christophersen~\cite{Behnke-Christophersen-1994} implies that:

\begin{proposition}\label{proposition:MilnorFiber-via-QBLDN}
Every Milnor fiber of a quotient surface singularity $X$ is symplectomorphic to the symplectic $4$-manifold that is obtained by rationally blowing down along singularities of class $T_0$ on the corresponding $M$-resolution of $X$.
\end{proposition}

We will end this subsection by discussing some simple necessary conditions for being $M$-resolutions. At first, every $(-1)$-curve on $Z$ (if any) should pass through at least two singularities of class $T_0$ on $Z$ because $K_Z$ is nef relative to $f$.

\begin{proposition}\label{proposition:reduction}
If $[a_1+1,a_2,\dotsc,a_r,2]-1-[b_1,\dotsc,b_t,2]$ is an $M$-resolution, then so is $[a_1,\dotsc,a_r]-1-[b_1,\dotsc,b_t]$.
\end{proposition}

\begin{proof}
See the proof of Lemma 9.2 in PPSU~\cite{PPSU-2018}.
\end{proof}

\begin{lemma}\label{lemma:b1>=a1}
Suppose that $[a_1, \dotsc, a_r]-1-[b_1,\dotsc,b_s]$ is a part of an $M$-resolution $Z$ of a quotient surface singularity $X$. Then $b_1 \ge a_1$.
\end{lemma}

\begin{proof}
If $a_1=2$, then the assertion is obvious. Assume that $a_1 \ge 3$.

If $[a_1,\dotsc,a_r] \neq [a_1,2,\dotsc,2]$, then for $i=r-a_1+2$, $a_i \ge 3$ and $a_{i+1}=\dotsb=a_r=2$. Contracting $(-1)$-curves successively on the minimal resolution $\widetilde{Z}$ of $Z$, we get a linear chain
\[a_1-a_2-\dotsb-(a_i-1)-(b_1-a_1+1)-b_2-\dotsb-b_r.\]
The above linear chain can be obtained by blowing up the minimal resolution $\widetilde{X}$ of $X$. So $-(b_1-a_1+1)$ cannot be nonnegative. Therefore $b_1 \ge a_1$.

If $[a_1,\dotsc,a_r] = [a_1,2,\dotsc,2]$, then $a_1 \ge 4$, $r=a_1-3$, and $a_2=\dotsb=a_r=2$. We then get a linear chain
\[(a_1-1)-(b_1-a_1+3)-b_2-\dotsb-b_r\]
after contracting $(-1)$-curves on $\widetilde{Z}$. Then $b_1-a_1+3 \ge 1$. Therefore there are at least $(a_1-4)$ 2's in the end of the sequence $b_1,\dotsc,b_t$; that is, $b_{j+1}=\dotsb=b_t=2$ for $j=t-a_1+4$.

Applying Proposition~\ref{proposition:reduction} repeatedly to  $[a_1,2, \dotsc, 2]-1-[b_1,\dotsc,b_s]$, we the have an $M$-resolution containing
\[[4]-1-[b_1-c_1+3,d_2,\dotsc,d_j].\]
Then $b_1-c_1+3 \ge 3$ because $K_Z$ is nef. For this, see the last paragraph at p.1211 in the proof of Lemma 9.2 in PPSU~\cite{PPSU-2018}.
\end{proof}

\begin{proposition}[{PPSU~\cite[Lemma~9.3]{PPSU-2018}}]\label{proposition:initial-curve}
Let $Z=[a_1,\dotsc,a_r]-1-[b_1,\dotsc,b_s]$ be an $M$-resolution of a cyclic quotient surface singularity $X$. Let $\widetilde{Z} \to Z$ be the minimal resolution of the singularities of class $T_0$ and let $\widetilde{X} \to X$ be the minimal resolution of $X$. Finally, let $f \colon \widetilde{Z} \to \widetilde{X}$ be the corresponding induced map. Then $f$ does not contract the initial curves of the singularities.
\end{proposition}

\subsection{Symplectic fillings as Milnor fibers}

Any minimal symplectic filling of a quotient surface singularity $(X,0)$ can be realized as the Milnor fiber of a smoothing $\pi \colon (X \subset \mathcal{X}) \to (0 \in \mathbb{D})$ of $X$.

According to Ohta-Ono~\cite{Ohta-Ono-2005}, NPP~\cite{Nemethi-PPampu-2010}, PPSU~\cite{PPSU-2018}, for any minimal symplectic filling $W$ of $(X,0)$, there is a smoothing $\mathcal{X} \to \Delta$ of $(X,0)$ such that $W$ is diffeomorphic to the Milnor fiber of the smoothing. As mentioned in the Introduction, each symplectic filling is given by complement of the compactifying divisor of the singularity in the rational symplectic 4-manifold $V$ and the symplectic deformation type of the filling is actually determined by homology data of the compactifying divisor in $V$; cf. Bhupal-Ono~\cite{Bhupal-Ono-2012} and Lisca~\cite{Lisca-2008}. One may check that the two homology data for a symplectic filling and its corresponding Milnor fiber coincide; cf. PPSU~\cite{PPSU-2018} for example. Hence:

\begin{proposition}
Let $(X,0)$ be a quotient surface singularity. For any minimal symplectic filling $W$ of $(X,0)$, there is a smoothing $\mathcal{X} \to \Delta$ of $(X,0)$ such that $W$ is symplectic deformation equivalent to the Milnor fiber of the smoothing.
\end{proposition}

Combined with Proposition~\ref{proposition:MilnorFiber-via-QBLDN}:

\begin{proposition}
Each minimal symplectic filling of a quotient surface singularity is symplectic deformation equivalent to the $4$-manifold that is obtained by rationally blowing down the corresponding $M$-resolution along the singularities of class $T_0$.
\end{proposition}

For any cyclic quotient surface singularity $X$, PPSU~\cite[\S 10.1]{PPSU-2018} provide an explicit algorithm for constructing the $M$-resolution of $X$ corresponding to a given minimal symplectic filling $W$ of $X$. Indeed minimal symplectic fillings of $X$ can be parametrized by certain sequences $\underline{k}=(k_1,\dotsc,k_e)$ representing zero Hirzebruch-Jung continued fraction. In PPSU~\cite[\S 10.1]{PPSU-2018}, they find a way to construct a $M$-resolution of $X$ corresponding to a given sequence $\underline{k}$. On the other hand, HJS~\cite{Han-Jeon-Shin-2018} make up the list of all $P$-resolutions of non-cyclic quotient surface singularities. So one may deal with $M$-resolutions instead of minimal symplectic fillings.

\section{Symplectic flips and antiflips}
\label{section:symplectic-flips}

We introduce a way to change over from a rationally blown-down regular neighborhood to another simpler one, which may be regarded as a symplectic version of the so-called \emph{usual flip} in Urz\'ua~\cite[Proposition~2.15]{Urzua-2013}.

\begin{proposition}\label{proposition:symplectic-flip}
Let $Y = [b_1, \dotsc, b_t]-1-c$. Suppose that $b_i \ge 3$ and $b_j=2$ for all $j > i$ (if any). If $i \ge 2$ and $c-b_1+1 \ge 1$, then $Y$ is symplectic deformation equivalent to $Y^+=b_1-[b_2,\dotsc,b_{i-1}, b_i-1]-(c-b_1+1)$. If $i=1$ and $c-b_1+3 \ge 1$, then $Y$ is symplectic deformation equivalent to $Y^+ = (b_1-1)-(c-b_1+3)$.
\end{proposition}

\begin{proof}
First we restrict ourselves to $i \geq 2$ and $c-b_1+1 \geq 2$ case. Since $[b_1,\dotsc, b_t]$ is of class $T_0$ and $b_j=2$ for any $j>i$, also $[b_1-(t-i),b_2, \dotsc, b_{i}]$ is of class $T_0$, consequently $b_1=(t-i)+2$. On the other hand, by contracting $(-1)$-curves successively, we get a linear chain $b_1-b_2- \dotsb -(b_i-1)-(c-(t-i+1))$ from $b_1-b_2- \dotsb- b_t-1-c$. Hence both $Y$ and $Y^+$ are symplectic fillings of the same cyclic quotient surface singularity.

Recall that each symplectic filling is given by complement of the compactifying divisor $C$ of the singularity and its symplectic deformation type is determined by the homology class of $C$ in rational symplectic $4$-manifold $Y\cong \mathbb{CP}^2 \# N \overline{\mathbb{CP}^2}$ which is obtained by gluing the compactifying divisor to the symplectic filling. For a linear chain $b_1-b_2-\dotsb-b_r$ with $b_i \geq 2$, the compactifying divisor $C$ is also a linear chain of $\mathbb{CP}^1$'s whose dual graph is
\begin{equation*}
\begin{tikzpicture}[scale=1]
\node[bullet] (00) at (0,0) [label=above:{$+1$}] {};
\node[bullet] (10) at (1,0) [label=above:{$1-a_1$}] {};
\node[bullet] (20) at (2,0) [label=above:{$-a_2$}] {};
\node[bullet] (30) at (3.5,0) [label=above:{$-a_e$}] {};
\node at (2.75,0)  {$\dotsb$};

\draw[-] (00)--(10);
\draw[-] (10)--(20);
\draw (2,0)--(2.5,0);
\draw (3,0)--(3.5,0);
\end{tikzpicture}
\end{equation*}
where $(a_1,\dotsc a_e)$ is dual Hirzebruch-Jung continued fraction of $(b_1,\dotsc, b_r)$ which means that $(b_1,\dotsc, b_r, 1, a_e, \dotsc, a_1)$ represents zero Hirzebruch-Jung continued fraction. Furthermore, embedding of $C$ to $Y$ is determined by certain $e$-tuples of integers $\underline{k}=(k_1,\dotsc,k_e)$ representing zero Hirzebruch-Jung continued fraction implying that the following linear chain
\begin{equation*}
\begin{tikzpicture}[scale=1]
\node[bullet] (00) at (0,0) [label=above:{$+1$}] {};
\node[bullet] (10) at (1,0) [label=above:{$1-k_1$}] {};
\node[bullet] (20) at (2,0) [label=above:{$-k_2$}] {};
\node[bullet] (30) at (3.5,0) [label=above:{$-k_e$}] {};
\node at (2.75,0)  {$\dotsb$};
\node at (4.5,0) {$(1\leq k_i \leq a_i)$};
\draw[-] (00)--(10);
\draw[-] (10)--(20);
\draw (2,0)--(2.5,0);
\draw (3,0)--(3.5,0);
\end{tikzpicture}
\end{equation*}
is obtained from
two distinct lines in $\mathbb{CP}^2$ by ordinary blowing ups. Then by blowing up $a_i-k_i$ distinct points at each vertex, we have an embedding of $K$. Therefore we could get explicit homology data of $K$ in $Y\cong \mathbb{CP}^2 \# N \overline{\mathbb{CP}^2}$ from the ways of blowing up to the above chain from two distinct lines in $\mathbb{CP}^2$; cf. Bhupal-Ono~\cite{Bhupal-Ono-2012} and Lisca~\cite{Lisca-2008}. Note that for a singularity $[b_1,\dotsc, b_r]$ of class $T_0$, an $e$-tuple $(k_1,\dotsc, k_e)$ corresponds to rational homology disk smoothing is obtained from $(2,1,2)$ by blowing ups at only intersection points on $(-1)$ curves so that there is only one $i$ such that $k_j=1$ and $a_j=2$. One can easily check that $k_i=a_i$ for $i\neq j$. In case of $(2,1,2)$, explicit homology data of corresponding linear chain in $\mathbb{CP}^2 \# 2 \overline{\mathbb{CP}^2}$ as follows. (Here $l$ denotes the homology class of $\mathbb{CP}^1$ in $\mathbb{CP}^2$ while $e_i$ denotes homology class of an exceptional curve.) %
\begin{equation*}
\begin{tikzpicture}[scale=1.5]
\node[bullet] (00) at (0,0) [label=above:{$+1$}] {};
\node[bullet] (10) at (1,0) [label=above:{$1-2$}] {};
\node[bullet] (20) at (2,0) [label=above:{$-1$}] {};
\node[bullet] (30) at (3,0) [label=above:{$-2$}] {};

\node[below] at (0,-0.1) {$l$};

\node[below,align=center] at (1,-.1) {$l-e_1-e_2$};
\node[below] at (2,-0.13) {$e_2$};
\node[below] at (3,-0.1) {{$e_1-e_2$}};
\draw[-] (00)--(30);

\end{tikzpicture}
\end{equation*}
Hence if $(k_1, \dotsc, k_e)$ is an $e$-tuple for rational homology disk smoothing for a singularity $[b_1,\dotsc, b_r]$ of class $T_0$, then  homology class of $i$th vertex in corresponding linear chain in $\mathbb{CP}^2 \# (e-1) \overline{\mathbb{CP}^2}$ is of the form $e_{i1}-\displaystyle\sum_{j=2}^{k_i} e_{ij}$ except the first and the second one where $e_{ij}\in \{e_1,\dotsc, e_{e-1}\}$. One can observe that homology class of the second vertex is of the form $l-e_{11}-e_{12}-\dotsb-{e_{1k_1}}$ and a homology class $e_{11}=e_{e1}=e_i$ does not appear in vertices other than the second and the last one.
\begin{equation*}
\begin{tikzpicture}[scale=1.5]
\node[bullet] (00) at (0,0) [label=above:{$+1$}] {};
\node[bullet] (10) at (1,0) [label=above:{$1-k_1$}] {};
\node[bullet] (20) at (2,0) [label=above:{$-k_2$}] {};
\node[bullet] (30) at (3.5,0) [label=above:{$-k_e$}] {};

\node[below] at (0,-0.1) {$l$};

\node[below,align=center] at (1,-.1) {$l-e_{11}-\dotsb-e_{1k_1}$};
\node[below] at (3.5,-0.1) {{$e_{e1}-\dotsb-e_{ek_e}$}};
\node at (2.75,0){$\dotsb$};
\draw[-] (0,0)--(2.5,0);
\draw (3,0)--(3.5,0);
\end{tikzpicture}
\end{equation*}
Furthermore, $(2, k_1,\dotsc, k_e+1)$ corresponds to rational homology disk smoothing for $[b_1+1, \dotsc, b_r, 2]$ and $(k_1+1, \dotsc, k_e, 2)$ corresponds to rational homology disk smoothing for $[2, b_1, \dotsc, b_r+1]$ and homology data of  corresponding linear chain in $\mathbb{CP}^2 \# e\overline{\mathbb{CP}^2}$ changes as follows. Here homology class $e_{ij}\in \{e_1,\dotsc, e_{e-1}\}$ comes from homology data of linear chain corresponds to $(k_1,\dotsc, k_e)$ in $\mathbb{CP}^2 \#(e-1)\overline{\mathbb{CP}^2}$.
\begin{equation}\label{operation1}
\begin{tikzpicture}[scale=2.3]
\node[bullet] (00) at (0,0) [label=above:{$+1$}] {};
\node[bullet] (10) at (1,0) [label=above:{$1-2$}] {};
\node[bullet] (20) at (2,0) [label=above:{$-k_1$}] {};
\node[bullet] (30) at (3.5,0) [label=above:{$-(k_e+1)$}] {};

\node[below] at (0,-0.1) {$l$};

\node[below,align=center] at (1,-.1) {$l-e_{11}-E$};
\node[below,align=center] at (2,-.1) {$E-e_{12}-\dotsb-e_{1k_1}$};

\node[below] at (3.5,-0.1) {{$e_{e1}-\dotsb-e_{ek_e}-E$}};
\node at (2.75,0){$\dotsb$};
\draw[-] (0,0)--(2.5,0);
\draw (3,0)--(3.5,0);
\end{tikzpicture}
\end{equation}
\begin{equation}\label{operation2}
\begin{tikzpicture}[scale=2.3]
\node[bullet] (00) at (0,0) [label=above:{$+1$}] {};
\node[bullet] (10) at (1,0) [label=above:{$1-(k_1+1)$}] {};
\node[bullet] (20) at (2,0) [label=above:{$-k_2$}] {};
\node[bullet] (30) at (3.5,0) [label=above:{$-2$}] {};

\node[below] at (0,-0.1) {$l$};

\node[below,align=center] at (1,-.1) {$l-E-e_{11}-\dotsb-e_{1k_1}$};
\node[below] at (3.5,-0.1) {{$E-e_{e1}$}};
\node at (2.75,0){$\dotsb$};
\draw[-] (0,0)--(2.5,0);
\draw (3,0)--(3.5,0);
\end{tikzpicture}
\end{equation}

Let $(a_1, \dotsc, a_e)$ be dual continued fraction of $b_2-\dotsb-b_{i-1}$. Then dual graph of the compactifying divisor $C$ of a cyclic quotient surface singularity $X$ determined by a linear chain $b_1-b_2-\dotsb-(b_{i-1}-1)-(c-b_1+1)$ and homology data for the minimal resolution of $X$ in rational symplectic $4$-manifold $Y\cong \mathbb{CP}\# (c+e+i-2) \overline{\mathbb{CP}^2}$ can be given as follows (Here $e_i$ and $E_j$ denote homology class of exceptional spheres.)
\begin{equation*}
\begin{tikzpicture}[scale=1.2]
\node[bullet] (00) at (0,0) [label=above:{$1$}] [label=below:{$l$}] {};
\node[bullet] (10) at (1,0) [label=above:{$1-2$}] [label=below:{$l-E_1-E_2$}]{};
\node[bullet] at (2,0) [label=above:{$-2$}] [label=below:]{};
\node[empty]  at (2.75,0) [] {$\dotsb$};
\node[bullet] at (3.5,0) [label=above:{$-2$}] [label=below:{$E_{b_1-2}-E_{b_1-1}$}]{};

\node[bullet] at (4.5,0) [label=above:{$-a_1-1$}]{};
\node[bullet] at (5.5,0) [label=above:{$-a_2$}]{};

\node[empty] at (6.25,0) [label=above:{}]{$\dotsb$};
\node[bullet] at (7,0) [label=above:{$-a_e-1$}]{};
\node[bullet] at (8,0) [label=above:{$-2$}][label=below:{$E_{b_1}-E_{b_1+1}$}]{};
\node[empty] at (8.75,0) [label=above:]{$\dotsb$};
\node[bullet] at (9.5,0) [label=above:{$-2$}][label=below:{$E_{c-2}-E_{c-1}$}]{};

\draw (0,0)--(2.5,0);
\draw (3,0)--(6,0);
\draw (6.5,0)--(8.5,0);
\draw (9,0)--(9.5,0);
\draw [thick, decoration={brace,mirror,raise=1.5em}, decorate] (1,0)--(3.5,0)
node [pos=0.5,anchor=north,yshift=-1.55em] {$b_1-2$};
\draw [thick, decoration={brace,mirror,raise=1.5em}, decorate] (8,0)--(9.5,0)
node [pos=0.5,anchor=north,yshift=-1.55em] {$c-b_1-1$};
\end{tikzpicture}
\end{equation*}

\begin{equation*}
\begin{tikzpicture}[scale=1.3]
\node[bullet] (00) at (0,0) [label=above:{$1$}] [label=below:{$l$}] {};
\node[bullet] (10) at (1,0) [label=above:{$1-2$}] [label=below:{}]{};
\node[bullet] at (2,0) [label=above:{$-2$}] [label=below:]{};
\node[empty]  at (2.75,0) [] {$\dotsb$};
\node[bullet] at (3.5,0) [label=above:{$-2$}] [label=below:{}]{};

\node[bullet] at (4.5,0) [label=above:{$-a_1-1$}][label=below:{$E_{b_1-1}-e_1-\dotsb-e_{a_1}$}]{};
\node[bullet] at (5.5,0) [label=above:{$-a_2$}]{};

\node[empty] at (6.25,0) [label=above:{}]{$\dotsb$};
\node[bullet] at (7,0) [label=above:{$-a_e-1$}][label=below:{$e_{e+i-a_e}-\dotsb-e_{e+i-1}-E_{b_1}$}]{};
\node[bullet] at (8,0) [label=above:{$-2$}][label=below:{}]{};
\node[empty] at (8.75,0) [label=above:]{$\dotsb$};
\node[bullet] at (9.5,0) [label=above:{$-2$}][label=below:{}]{};

\draw (0,0)--(2.5,0);
\draw (3,0)--(6,0);
\draw (6.5,0)--(8.5,0);
\draw (9,0)--(9.5,0);
\end{tikzpicture}
\end{equation*}
while embedding of $b_2-\dotsb-(b_{i}-1)$ can be given as follows.
\begin{equation*}
\begin{tikzpicture}[scale=1.5]
\node[bullet] (00) at (0,0) [label=above:{$-b_2$}] {};
\node[bullet] (20) at (1.5,0) [label=above:{$-b_3$}] {};
\node[bullet] (30) at (4,0) [label=above:{$-b_{i}+1$}] {};

\node[below] at (0,-0.1) {$e_1-e_2-\dotsb-e_{b_2}$};

\node[below,align=center] at (1.5,-.1) {$e_{b_2}-\dotsb-e_{b_2+b_3-1}$};
\node[below] at (4,-0.1) {{$e_{e+i-b_i-3}-\dotsb-e_{e+i-1}$}};
\node at (2.75,0){$\dotsb$};
\draw[-] (0,0)--(2.5,0);
\draw (3,0)--(4,0);
\end{tikzpicture}
\end{equation*}

Suppose that $e$-tuple $(k_1,\dotsb, k_e)$ corresponds to rational homology disk smoothing of a cyclic quotient surface singularity $X'$ determined by $b_2-\dotsb-(b_i-1)$ and homology data for $e$-tuple in $\mathbb{CP}^2 \# (e-1) \overline{\mathbb{CP}^2}$ is given as follows.
\begin{equation*}
\begin{tikzpicture}[scale=1.5]
\node[bullet] (00) at (0,0) [label=above:{$+1$}] {};
\node[bullet] (10) at (1,0) [label=above:{$1-k_1$}] {};
\node[bullet] (20) at (2,0) [label=above:{$-k_2$}] {};
\node[bullet] (30) at (3.5,0) [label=above:{$-k_e$}] {};

\node[below] at (0,-0.1) {$l$};

\node[below,align=center] at (1,-.1) {$l-e_{11}-\dotsb-e_{1k_1}$};
\node[below] at (3.5,-0.1) {{$e_{e1}-\dotsb-e_{ek_e}$}};
\node at (2.75,0){$\dotsb$};
\draw[-] (0,0)--(2.5,0);
\draw (3,0)--(3.5,0);
\end{tikzpicture}
\end{equation*}
Then homology data of the compactifying divisor $C'$ corresponds to rational homology disk smoothing of $[b_2,\dotsc ,b_{i}-1]$ in $\mathbb{CP}^2\# e\overline{\mathbb{CP}^2}$ can be given as follows (Here $e_e$ represents homology class of exceptional curve coming from blow up from $\mathbb{CP}^2\# (e-1)\overline{\mathbb{CP}^2}$ to $\mathbb{CP}^2\# e\overline{\mathbb{CP}^2}$.)
\begin{equation*}
\begin{tikzpicture}[scale=1.5]
\node[bullet] (00) at (0,0) [label=above:{$+1$}] {};
\node[bullet] (10) at (1,0) [label=above:{$1-a_1$}] {};
\node[bullet] (20) at (2,0) [label=above:{$-a_2$}] {};
\node[bullet] (30) at (3.5,0) [label=above:{$-a_j$}] {};
\node[bullet]  at (5,0) [label=above:{$-a_e$}] {};

\node[below] at (0,-0.1) {$l$};

\node[below,align=center] at (1,-.1) {$l-e_{11}-\dotsb-e_{1k_1}$};
\node[below] at (3.5,-0.1) {{$e_{j1}-e_e$}};

\node[below] at (5,-0.1) {{$e_{e1}-\dotsb-e_{ek_e}$}};
\node at (2.75,0){$\dotsb$};

\node at (4.25,0){$\dotsb$};
\draw[-] (0,0)--(2.5,0);
\draw (3,0)--(4,0);
\draw (4.5,0)--(5,0);

\end{tikzpicture}
\end{equation*}
while homology data of $C'$ for the minimal resolution in $\mathbb{CP}^2 (e+i-1) \overline{\mathbb{CP}^2}$ can be given as follows.
\begin{equation*}
\begin{tikzpicture}[scale=1.5]
\node[bullet] (00) at (0,0) [label=above:{$+1$}] {};
\node[bullet] (10) at (1,0) [label=above:{$1-a_1$}] {};
\node[bullet] (20) at (2,0) [label=above:{$-a_2$}] {};
\node[bullet] (30) at (3.5,0) [label=above:{$-a_e$}] {};

\node[below] at (0,-0.1) {$l$};

\node[below,align=center] at (1,-.1) {$l-e_{1}-\dotsb-e_{a_1}$};
\node[below] at (3.5,-0.1) {{$e_{e+i-a_e}-\dotsb-e_{e+i-1}$}};
\node at (2.75,0){$\dotsb$};
\draw[-] (0,0)--(2.5,0);
\draw (3,0)--(3.5,0);
\end{tikzpicture}
\end{equation*}
Hence if we rationally blow down $[b_2,\dotsc ,b_i-1]$ from the linear chain $b_1-\dotsb-(b_i-1)-(c-b_1+1)$ , then we get new rational symplectic $4$-manifold $Y'\cong \mathbb{CP}^2 (c+e-1) \overline{\mathbb{CP}^2}$ and homology data of compactifying divisor $C$ in $Y'$ changes as follows.
\begin{equation*}
\begin{tikzpicture}[scale=1.2]
\node[bullet] (00) at (0,0) [label=above:{$1$}] [label=below:{$l$}] {};
\node[bullet] (10) at (1,0) [label=above:{$1-2$}] [label=below:{$l-E_1-E_2$}]{};
\node[bullet] at (2,0) [label=above:{$-2$}] [label=below:]{};
\node[empty]  at (2.75,0) [] {$\dotsb$};
\node[bullet] at (3.5,0) [label=above:{$-2$}] [label=below:{$E_{b_1-2}-E_{b_1-1}$}]{};

\node[bullet] at (4.5,0) [label=above:{$-a_1-1$}]{};
\node[bullet] at (5.5,0) [label=above:{$-a_2$}]{};

\node[empty] at (6.25,0) [label=above:{}]{$\dotsb$};
\node[bullet] at (7,0) [label=above:{$-a_e-1$}]{};
\node[bullet] at (8,0) [label=above:{$-2$}][label=below:{$E_{b_1}-E_{b_1+1}$}]{};
\node[empty] at (8.75,0) [label=above:]{$\dotsb$};
\node[bullet] at (9.5,0) [label=above:{$-2$}][label=below:{$E_{c-2}-E_{c-1}$}]{};

\draw (0,0)--(2.5,0);
\draw (3,0)--(6,0);
\draw (6.5,0)--(8.5,0);
\draw (9,0)--(9.5,0);
\draw [thick, decoration={brace,mirror,raise=1.5em}, decorate] (1,0)--(3.5,0)
node [pos=0.5,anchor=north,yshift=-1.55em] {$b_1-2$};
\draw [thick, decoration={brace,mirror,raise=1.5em}, decorate] (8,0)--(9.5,0)
node [pos=0.5,anchor=north,yshift=-1.55em] {$c-b_1-1$};
\end{tikzpicture}
\end{equation*}

\begin{equation*}
\begin{tikzpicture}[scale=1.45]
\node[bullet] (00) at (1,0) [label=above:{$1$}] [label=below:{$l$}] {};
\node[bullet] (10) at (1.5,0) [label=above:{$1-2$}] [label=below:{}]{};
\node[bullet] at (2,0) [label=above:{$-2$}] [label=below:]{};
\node[empty]  at (2.5,0) [] {$\dotsb$};
\node[bullet] at (3,0) [label=above:{$-2$}] [label=below:{}]{};

\node[bullet] at (4,0) [label=above:{$-a_1-1$}][label=below:{$E_{b_1-1}-\textcolor{red}{e_{11}-\dotsb-e_{1k_1}}$}]{};

\node[empty] at (4.75,0) [label=above:{}]{$\dotsb$};

\node[bullet] at (5.5,0) [label=above:{$-a_j$}][label=below:{\textcolor{red}{$e_{j1}-e_{e}$}}]{};

\node[empty] at (6.25,0) [label=above:{}]{$\dotsb$};

\node[bullet] at (7.,0) [label=above:{$-a_e-1$}][label=below:{\textcolor{red}{$e_{e1}-\dotsb-e_{ek_e}$}$-E_{b_1}$}]{};
\node[bullet] at (8.,0) [label=above:{$-2$}][label=below:{}]{};
\node[empty] at (8.5,0) [label=above:]{$\dotsb$};
\node[bullet] at (9.,0) [label=above:{$-2$}][label=below:{}]{};

\draw (1,0)--(2.25,0);
\draw (2.75,0)--(4.5,0) (5.,0)--(6.,0);
\draw (6.5,0)--(8.25,0) (8.75,0)--(9,0);
\end{tikzpicture}
\end{equation*}

Note that the changes only occur at vertices with $e_i$. To get homology data of $C$ for $Y$, we start from a linear chain $b_1-b_2-\dotsb-(b_i-1)-c$. As a symplectic filling of $X$, homology data of the linear chain itself is same as the minimal resolution of $X$ but ambient rational symplectic $4$-manifold has changed from $Y\cong \mathbb{CP}^2\# (c+e+i-2) \overline{\mathbb{CP}^2}$ to $Y\#(t-i+1)\overline{\mathbb{CP}^2}\cong \mathbb{CP}^2\# (c+e+t-1) \overline{\mathbb{CP}^2}$ so that homology data of embedding of the linear chain can be given as follows. (Here $E'_i$ denote homology class of exceptional spheres from $(t-i+1)\overline{\mathbb{CP}^2}$ in $Y\# (t-i+1)\overline{\mathbb{CP}^2}$.)
\begin{equation*}
\begin{tikzpicture}[scale=1.2]
\node[bullet] (00) at (0,0) [label=above:{$-b_1$}][label=below:{$E_1-\dotsb-E_{b_1-1}-e_1$}] {};
\node[bullet] (20) at (2,0) [label=above:{$-b_2$}][label=below:{$e_1-\dotsb-e_{b_2}$}] {};
\node[bullet] (30) at (4,0) [label=above:{$-b_{i}$}][label=below:{$e_{e+i-b_i-3}-\dotsb-e_{e+i-1}-E'_{1}$}]{};
\node[bullet] (30) at (6,0) [label=above:{$-b_{i+1}=-2$}][label=below:{$E'_{1}-E'_{2}$}] {};
\node[bullet] (30) at (8,0) [label=above:{$-b_{t}=-2$}][label=below:{$E'_{t-i}-E'_{t-i+1}$}] {};

\node at (3,0){$\dotsb$};
\node at (7,0){$\dotsb$};

\draw[-] (0,0)--(2.5,0);
\draw (3.5,0)--(6.5,0) (7.5,0)--(8,0);

\end{tikzpicture}
\end{equation*}

On the other hand, using operation \ref{operation1} and \ref{operation2} in the previous paragraph, we could get $(e+b_1-1)$-tuple $(2,\dotsc, 2, k_1+1, k_2, \dotsc, k_e, b_1)$ corresponds to rational homology disk smoothing $[b_1,\dotsc, b_t]$ keeping track of homology data in terms of homology data of $(k_1,\dotsc, k_e)$ in $\mathbb{CP}^2 \# (e-1)\overline{\mathbb{CP}^2}$. Explicitly, we have embedding of following linear chain to $\mathbb{CP}^2\# (e-1) \overline{\mathbb{CP}^2} \# (b_1-1)\overline{\mathbb{CP}^2}$ where $E_{i}$ comes from $(b_1-1)\overline{\mathbb{CP}^2}$ and $e_{ij}$ comes from $(e-1)\overline{\mathbb{CP}^2}$ as before.

\begin{equation*}
\begin{tikzpicture}[scale=1.3]
\node[bullet] (00) at (0,0) [label=above:{$1$}] [label=below:{$l$}] {};
\node[bullet] (10) at (1,0) [label=above:{$1-2$}] [label=below:{$l-E_1-E_2$}]{};
\node[bullet] at (2,0) [label=above:{$-2$}] [label=below:]{};
\node[empty]  at (2.75,0) [] {$\dotsb$};
\node[bullet] at (3.5,0) [label=above:{$-2$}] [label=below:{$E_{b_1-2}-E_{b_1-1}$}]{};

\node[bullet] at (4.5,0) [label=above:{$-k_1-1$}]{};
\node[bullet] at (5.5,0) [label=above:{$-k_2$}]{};

\node[empty] at (6.25,0) [label=above:{}]{$\dotsb$};
\node[bullet] at (7,0) [label=above:{$-k_e$}]{};
\node[bullet] at (8,0) [label=above:{$-b_1$}][label=below:{$E_{1}-\dotsb-E_{b_1-1}-e_{e1}$}]{};

\draw (0,0)--(2.5,0);
\draw (3,0)--(6,0);
\draw (6.5,0)--(8,0);
\draw [thick, decoration={brace,mirror,raise=1.5em}, decorate] (1,0)--(3.5,0)
node [pos=0.5,anchor=north,yshift=-1.55em] {$b_1-2$};

\end{tikzpicture}
\end{equation*}

\begin{equation*}
\begin{tikzpicture}[scale=1.45]
\node[bullet] (00) at (1,0) [label=above:{$1$}] [label=below:{$l$}] {};
\node[bullet] (10) at (1.5,0) [label=above:{$1-2$}] [label=below:{}]{};
\node[bullet] at (2,0) [label=above:{$-2$}] [label=below:]{};
\node[empty]  at (2.5,0) [] {$\dotsb$};
\node[bullet] at (3,0) [label=above:{$-2$}] [label=below:{}]{};

\node[bullet] at (4,0) [label=above:{$-k_1-1$}][label=below:{$E_{b_1-1}-\textcolor{black}{e_{11}-\dotsb-e_{1k_1}}$}]{};

\node[empty] at (4.75,0) [label=above:{}]{$\dotsb$};

\node[bullet] at (5.5,0) [label=above:{$-k_j$}][label=below:{\textcolor{black}{$e_{j1}$}}]{};

\node[empty] at (6.25,0) [label=above:{}]{$\dotsb$};

\node[bullet] at (7.,0) [label=above:{$-k_e$}][label=below:{\textcolor{black}{$e_{e1}-\dotsb-e_{ek_e}$}}]{};
\node[bullet] at (8.,0) [label=above:{$-b_1$}][label=below:{}]{};

\draw (1,0)--(2.25,0);
\draw (2.75,0)--(4.5,0) (5.,0)--(6.,0);
\draw (6.5,0)--(8,0) ;
\end{tikzpicture}
\end{equation*}

The above observation implies that if we rationally blow down $[b_1, \dotsc , b_t]$ from $Y\#(t-i+1)\overline{\mathbb{CP}^2}\cong \mathbb{CP}^2\# (c+e+t-1) \overline{\mathbb{CP}^2}$, we get a rational symplectic manifold $Y''\cong \mathbb{CP}^2 (c+e-1) \overline{\mathbb{CP}^2}$ and homology data of $C$ in $Y''$ can be given the same as $C$ in $Y'$ which means that $Y$ is symplectic deformation equivalent to $Y^+$. The situation is different for $c-b_1+1=1$ case, because linear chain $b_1-b_2-\dotsb-(b_i-1)-(c-b_1+1)$ is not minimal. Instead, if we start from $b_1-b_2-\dotsb-(b_k-1)$ which is obtained by contracting $(-1)$ curves successively, then the similar argument shows that $Y$ and $Y^+$ are symplectic deformation equivalent. In case of $i=1$, similar argument as for $i\geq 2$ case shows that $b_1=t+3$ and $(b_1-1)-(c-b_1+3)$ is obtained from $b_1-\dotsb-b_t-1-c$ by contracting $(-1)$ curves successively so that $Y$ and $Y^+$ are symplectic fillings of same cyclic quotient surface singularity $X$. One can easily compute homology data of $C$ for $Y$ using following homology data for rational homology disk smoothing of $[b_1,\dotsc, b_t]=[t+3,2,\dotsc, 2]$ so that it is same as that of $C$ for linear chain $(b_1-1)-(c-b_1+3)$ which is equal to $Y^+$.
\begin{equation*}
\begin{tikzpicture}[scale=2]
\node[bullet] (00) at (0,0) [label=above:{$1$}] [label=below:{$l$}] {};
\node[bullet] (10) at (1,0) [label=above:{$1-2$}] [label=below:{$l-e_1-e_2$}]{};
\node[bullet] (250) at (2.5,0) [label=above:{$-2$}] [label=below:{$e_t-e_{t+1}$}]{};
\node[empty] (20) at (2,0) [] {};
\node[empty] (150) at (1.5,0) [] {};

\node[bullet] (350) at (3.5,0) [label=above:{$-t$}] [label=below:{$e_1-\cdots-e_{t}$}]{};

\draw[-] (00)--(10);
\draw[-] (10)--(150);
\draw[dotted] (150)--(20);
\draw[-] (20)--(250);
\draw[-] (250)--(350);
\draw [thick, decoration={brace,mirror,raise=1.5em}, decorate] (10)--(250)
node [pos=0.5,anchor=north,yshift=-1.55em] {$t$};
\end{tikzpicture}
\end{equation*}
 \end{proof}

%
%
%

\begin{definition}[Symplectic flips and antiflips]
A \emph{symplectic flip} is the operation that transforms $Y=[b_1, \dotsc, b_t]-1-c$ to $Y^+$ as the above Proposition~\ref{proposition:symplectic-flip}. On the other hand, a \emph{symplectic antiflip} is the converse operation to a symplectic flip; that is, it transforms $Y^+$ to $Y$.
\end{definition}

\section{Cyclic quotient surface singularities}
\label{section:cyclic}

We prove Theorem~\ref{theorem:main} for cyclic quotient surface singularities; Theorem~\ref{theorem:main-cyclic}.

Let $(X,0)$ be a cyclic quotient surface singularity and let $\widetilde{X} \to X$ be its minimal resolution. Let $W$ be a minimal symplectic filling of $X$. Let $Z \to X$ be the $M$-resolution corresponding to $W$. We denote again by $Z$ a general fiber of the $\mathbb{Q}$-Gorenstein smoothing of $Z$ if no confusion arises. That is, one can say that $W$ is symplectic deformation equivalent to $Z$.

We begin with the simplest case.

\begin{lemma}\label{lemma:n=2}
If an $M$-resolution $Z$ of $X$ is of the form $[a_1,\dotsc,a_r]-1-[b_1,\dotsc,b_s]$, then there is a sequence of rational blow-ups and symplectic flips that transforms $Z$ into $\widetilde{X}$.
\end{lemma}

\begin{proof}
We first rationally blow up $Z$ so that we have
\begin{equation*}
Z' := [a_1,\dotsc,a_r]-1-b_1-\dotsb-b_s.
\end{equation*}

If $[a_1,\dotsc,a_r]=[a_1,2\dotsc,2]$, then, by applying a symplectic flip to $Z'$, we get $\widetilde{X}=(a_1-1)-(b_1-a_1+3)-b_2-\dotsb-b_s$.

Suppose now that $[a_1,\dotsc,a_r] \neq [a_1,2,\dotsc,2]$. After symplectically flipping $Z'$, we have
\begin{equation*}
Z'^+ := a_1-[a_2,\dotsc,a_{i-1},a_i-1]-(b_1-a_1+1)-b_2-\dotsc-b_s,
\end{equation*}
where $a_i \ge 3$ but $a_{i+1}=\dotsb=a_r=2$.

Note that $b_1-a_1+1 \ge 1$ be Lemma~\ref{lemma:b1>=a1}. If $b_1-a_1+1 \neq 1$, we rationally blow up $Z'^+$ along $[a_2,\dotsc,a_{i-1},a_i-1]$ so that we get
\begin{equation*}
\widetilde{X}=a_1-a_2-\dotsb-a_{i-1}-(a_i-1)-(b_1-a_1+1)-b_2-\dotsc-b_s.
\end{equation*}

On the other hand, if $b_1-a_1+1=1$, then we can repeat symplectic flips again. This process should stop after finitely many flips because the initial curve of $[b_1,\dotsc,b_t]$ cannot be killed by flips (which are just ordinary blow-downs in the level of resolutions) by Proposition~\ref{proposition:initial-curve}. That is, if $b_j$ is the initial curve of $[b_1,\dotsc,b_t]$, then its proper transforms after symplectic flips cannot be a $(-1)$-curve.

Therefore, after finitely many symplectic flips, we may end up with a configuration $Z^+$ without $(-1)$-curves. Then we get $\widetilde{X}$ by rationally blowing up $Z^+$ along singularities of class $T_0$ (if any).
\end{proof}

\begin{remark}\label{remark:keep-the-other-parts}
Let $a_i$ and $b_j$ be the initial curves of $[a_1,\dotsc,a_r]$ and $[b_1,\dotsc,b_s]$, respectively. The initial curves $a_i$ and $b_j$ are not killed by symplectic flips according to Proposition~\ref{proposition:initial-curve}. So $\widetilde{X}$ contains linear chains of the form $a_1-\dotsc-a_{i-1}-a_i'$ and $b_j'-b_{j+1}-\dotsb-b_s$, where $a_i'$ and $b_j'$ are proper transforms of $a_i$ and $b_j$, respectively. Notice that $a_i', b_j' \ge 2$. In short, any sequence of symplectic flips that transforms $Z$ to $\widetilde{X}$ occurs between $a_i$ and $b_j$ and do not alter the other parts.
\end{remark}

Similarly, we don't need the whole curves in $[b_1,\dotsc,b_t]$ in order to transform $Z$ to $\widetilde{X}$. That is:

\begin{corollary}\label{corollary:n=2}
let $Z := [a_1,\dotsc,a_r]-1-[b_1,\dotsc,b_t]$ be an $M$-resolution of a cyclic quotient surface singularity $X$. Let $b_j$ be the initial curve of $[b_1,\dotsc,b_t]$. Set $b_j'=b_j$ or $b_j-1$. Let $X''$ is a new cyclic quotient surface singularity obtained by contracting $a_1-\dotsb-a_r-1-b_1-\dotsb-b_j'$ and let $Z'' := [a_1,\dotsc,a_r]-1-b_1-\dotsb-b_j'$. Then there is a sequence of symplectic flips that transforms $Z''$ to the minimal resolution of $X''$.
\end{corollary}

\begin{proof}
A sequence of symplectic flips in the above lemma could not kill the initial curve $b_j$, that is, the proper transform of the initial curve $b_j$ cannot be a $(-1)$-curve. So a sequence of symplectic flips that transforms $Z$ to $\widetilde{X}$ occurs from $b_1$ (possibly) upto the curve $b_{j-1}$ right before the initial curve $b_j$. Hence the assertion follows.
\end{proof}

\begin{theorem}\label{theorem:main-cyclic}
Let $Z$ be an $M$-resolution of a cyclic quotient surface singularity $X$. Then there is a sequence of rational blow-downs and symplectic antiflips that transforms $\widetilde{X}$ to $Z$.
\end{theorem}

\begin{proof}
We will prove that there is a sequence of rational blow-ups and symplectic flips that transforms $Z$ to $\widetilde{X}$.

Suppose that there are \emph{isolated} singularities of class $T_0$ on $Z$, that is, the singularities are not connected by $(-1)$-curves with another ones. Then we rationally blow up $Z$ along the isolated singularities. We now assume that there are non-isolated singularities of class $T_0$ on $Z$. Suppose that $Z$ contains a linear chain
\[[P_1]-1-\dotsb-1-[P_n]\]
of non-isolated singularities $[P_i]$ of length $n \ge 2$.

During the proof below, we will show that any sequence of rational blow ups and symplectic flips that shall be introduced does not alter the other parts of $Z$; cf.~Remark~\ref{remark:keep-the-other-parts}. So we may assume that $Z$ itself is given by $[P_1]-1-\dotsb-1-[P_n]$ for simplicity. Then $Z$ is of the form
\[Z=L-[a_1,\dotsc,a_r]-1-[b_1,\dotsc,b_s]\]
where $L$ is the (possibly empty) leftmost part of $Z$.

We first rationally blow up $Z$ along $[b_1,\dotsc,b_s]$ to get
\[Z'=L-[a_1,\dotsc,a_r]-1-b_1-\dotsb-b_s\]
We claim that there is a sequence of symplectic flips that transforms $Z'$ to $\widetilde{X}$. We use an induction on the number $n-1$ of singularities of class $T_0$ on $Z'$.

Let $Y=[a_1,\dotsc,a_r]-1-b_1-\dotsb-b_s$. By Corollary~\ref{corollary:n=2} we symplectically flip $Y$ (repeatedly if necessary) so that we get a minimal resolution $\widetilde{Y}$ that is obtained by contracting $(-1)$-curves (successively if necessary) from $a_1-\dotsb-a_r-1-b_1-\dotsb-b_s$. Then we have
\[Z'' = L-\widetilde{Y}\]

If $L$ is empty, then we are done. If $L$ is not empty, $Z'^+$ is of the form
\[Z''=L'-[c_1,\dotsc,c_t]-1-\widetilde{Y}\]
Note that the minimal resolution $\widetilde{Y}$ contains the initial curve of $[a_1,\dotsc,a_r]$; Remark~\ref{remark:keep-the-other-parts}. So we can repeat the above process to $Z''$ again. Then the assertion follows by induction.
\end{proof}

\subsection{Examples}\label{subsection:Examples}

\subsubsection{}

Let $X$ be a cyclic quotient surface singularity of type $\frac{1}{37}(1,10)$, whose dual graph of the minimal resolution $\widetilde{X}$ is given by $4-4-2-2$. Let $W$ be a minimal symplectic filling of $X$ corresponding to the $M$-resolution $Z=[5,2]-1-[6,2,2]$. The sequence of rational blow-downs and symplectic antiflips are as follows:

\begin{enumerate}
\item[0.] The minimal resolution $\widetilde{X}$: $4-4-2-2$

\item[1.] A symplectic antiflip along $4-4$: $[5,2]-1-6-2-2$

\item[2.] A rational blow-down along $6-2-2$: $[5,2]-1-[6,2,2]$
\end{enumerate}

We need a rational blow-down along $6-2-2$ in the process to get $W$ from $\widetilde{X}$. It is hidden in the configuration $4-4$, whose symplectic sum gives us the desired $(-6)$-curve. So the symplectic antiflip along $4-4$ shows us clearly that there is a configuration $6-2-2$ where we have to rational blow down.

\subsubsection{}

Let $X$ be a cyclic quotient surface singularity of type $\frac{1}{81}(1, 47)$. The dual graph of the minimal resolution $\widetilde{X}$ is given by $2-4-3-3-2$. Let $W$ be a symplectic filling of $X$ corresponding to the $M$-resolution $Z=[2,5,3]-1-[2,5,3]-2$.

A sequence which transforms $\widetilde{X}$ to $W$ is as follows:

\begin{enumerate}
\item[0.] The minimal resolution $\widetilde{X}$: $2-4-3-3-2$

\item[1.] A symplectic antiflip along $4-3$: $2-[5,2]-1-5-3-2$

\item[2.] A symplectic antiflip along $2-[5,2]-1$: $[2,5,3]-1-2-5-3-2$

\item[3.] A rational blow-down along $2-5-2$: $[2,5,3]-1-[2,5,3]-2$
\end{enumerate}

We apply only one rational blow-down along $2-5-3$. Here the $(-5)$-curve is given as a symplectic sum of $2-4-3$ as before. But it is not easy to pull out the $(-2)$-curve from $\widetilde{X}$. Once again, symplectic antiflips put the desired curve on our hands.

It would be an intriguing problem to compare the above algorithm with that given in Bhupal-Ozbagci~\cite[\S4.1]{Bhupal-Ozbagci-2016} where the symplectic filling $W$ is denoted by $W_{(81,47)}((3,2,1,3,2))$.

\subsubsection{}

Let $X$ be a cyclic quotient surface singularity of type $\frac{1}{45}(1,26)$, whose dual graph of the minimal resolution $\widetilde{X}$ is given by $2-4-4-2$. Let $W$ be a minimal symplectic filling of $X$ corresponding to the $M$-resolution $Z = [2,5]-1-[5,2]$. A sequence from $\widetilde{X}$ to $W$ is as follows:

\begin{enumerate}
\item[0.] The minimal resolution $\widetilde{X}$: $2-4-4-2$

\item[1.] A rational blow-down along $4$: $2-[4]-4-2$

\item[2.] A symplectic antiflip along $2-[4]-4$: $[2,5]-1-5-2$

\item[3.] A rational blow-down along $5-2$: $[2,5]-1-[5,2]$
\end{enumerate}

In this example we need two rational blow-downs: One along $4$ and the other along $5-2$. The first rational blow-down can be easily recognized in the minimal resolution $\widetilde{X}$. But the $(-5)$-curve for the second one is not found at a glance from the rationally blown-down $2-[4]-4-2$. So a symplectic antiflip again provides us the desired negative curves.

\section{Non-cyclic quotient surface singularities}
\label{section:non-cyclic}

There are four classes of non-cyclic quotient surface singularities: Dihedral singularities, tetrahedral singularities, octahedral singularities, icosahedral singularities. Symplectic fillings of dihedral singularities are essentially determined by that of cyclic quotient surface singularities; cf.~Bhupal-Ono~\cite{Bhupal-Ono-2012}. So one may apply the same algorithm described in the previous section for dihedral singularities. Hence we deal only tetrahedral, octahedral, icosahedral singularities (denoted by \emph{TOI-singularities} for short) in this section.

Let $Y$ be a TOI-singularity and let $\widetilde{Y}$ be its minimal resolution. For an $M$-resolution $Z$ of $Y$, let $\Gamma$ be a maximal connected subgraph of the dual graph of its minimal resolution $\widetilde{Z}$ that contains all the dual graphs of singularities of class $T_0$ of $Z$.

\textbf{Case 1}. $\Gamma$ is linear.

One may apply the same procedure for cyclic quotient surface singularities described in the previous section.

\textbf{Case 2}. $\Gamma$ is non-linear.

According to the list of $P$-resolutions of TOI-singularities in HJS-\cite{Han-Jeon-Shin-2018}, there are only 9 types of $\Gamma$ described in Figure~\ref{figure:subgraph-Gamma}. Then we provide the desired sequence of rational blow-downs case by case in Figure~\ref{figure:sequence-Gamma_1}--Figure~\ref{figure:sequence-Gamma_9}. For example, in case of $\Gamma_1$, the dual grpah of the minimal resolution $\widetilde{Y}$ should contain the subgraph
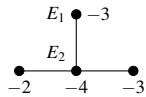
\begin{equation*}
\begin{tikzpicture}[scale=0.75]

\node[bullet] (01) at (0,1) [label=right:{$-3$},label=left:{$E_1$}] {};
\node[bullet] (-10) at (-1,0) [label=below:{$-2$}] {};
\node[bullet] (00) at (0,0) [label=below:{$-4$},label=135:{$E_2$}] {};
\node[bullet] (10) at (1,0) [label=below:{$-3$}] {};

\draw (01)--(00);
\draw (-10)--(00)--(10);
\end{tikzpicture}
\end{equation*}
We first symplectically antiflip $E_1$ and $E_2$. Then we have
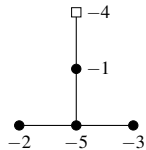
\begin{equation*}
\begin{tikzpicture}[scale=0.75]

\node[rectangle] (02) at (0,2) [label=right:{$-4$}] {};
\node[bullet] (01) at (0,1) [label=right:{$-1$}] {};
\node[bullet] (-10) at (-1,0) [label=below:{$-2$}] {};
\node[bullet] (00) at (0,0) [label=below:{$-5$}] {};
\node[bullet] (10) at (1,0) [label=below:{$-3$}] {};

\draw (02)--(01)--(00);
\draw (-10)--(00)--(10);
\end{tikzpicture}
\end{equation*}
We next rationally blow-down $2-5-3$ in the above configuration. Then we obtain the symplectic filling described by the graph $\Gamma$. We summarize this procedure in Figure~\ref{figure:sequence-Gamma_1}.

In the following Figures~\ref{figure:sequence-Gamma_2}--\ref{figure:sequence-Gamma_9} we describe the desired sequence of rationally blow-downs together with symplectic antiflips for each $\Gamma_i$'s.

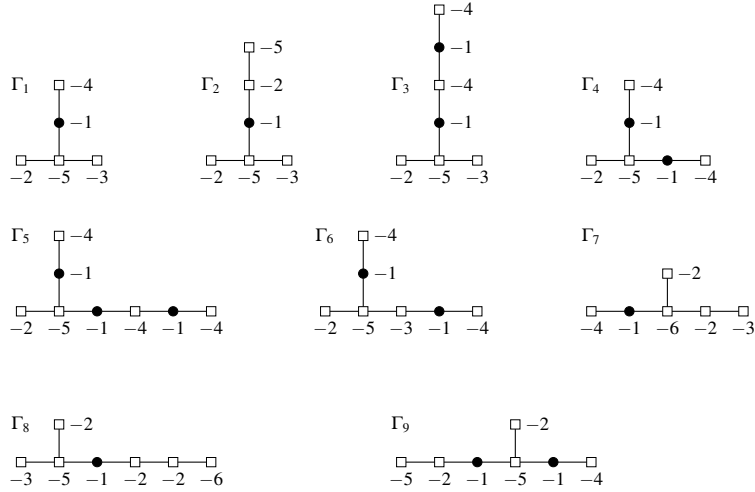
\begin{figure}
\centering
\begin{tikzpicture}[scale=0.5]

\begin{scope}
\node[empty] (-12) at (-1,2) [] {$\Gamma_1$};

\node[rectangle] (02) at (0,2) [label=right:{$-4$}] {};
\node[bullet] (01) at (0,1) [label=right:{$-1$}] {};
\node[rectangle] (-10) at (-1,0) [label=below:{$-2$}] {};
\node[rectangle] (00) at (0,0) [label=below:{$-5$}] {};
\node[rectangle] (10) at (1,0) [label=below:{$-3$}] {};

\draw (02)--(01)--(00);
\draw (-10)--(00)--(10);
\end{scope}

\begin{scope}[shift={(5,0)}]
\node[empty] (-12) at (-1,2) [] {$\Gamma_2$};

\node[rectangle] (03) at (0,3) [label=right:{$-5$}] {};
\node[rectangle] (02) at (0,2) [label=right:{$-2$}] {};
\node[bullet] (01) at (0,1) [label=right:{$-1$}] {};
\node[rectangle] (-10) at (-1,0) [label=below:{$-2$}] {};
\node[rectangle] (00) at (0,0) [label=below:{$-5$}] {};
\node[rectangle] (10) at (1,0) [label=below:{$-3$}] {};

\draw (03)--(02)--(01)--(00);
\draw (-10)--(00)--(10);
\end{scope}

\begin{scope}[shift={(10,0)}]
\node[empty] (-12) at (-1,2) [] {$\Gamma_3$};

\node[rectangle] (04) at (0,4) [label=right:{$-4$}] {};
\node[bullet] (03) at (0,3) [label=right:{$-1$}] {};
\node[rectangle] (02) at (0,2) [label=right:{$-4$}] {};
\node[bullet] (01) at (0,1) [label=right:{$-1$}] {};
\node[rectangle] (-10) at (-1,0) [label=below:{$-2$}] {};
\node[rectangle] (00) at (0,0) [label=below:{$-5$}] {};
\node[rectangle] (10) at (1,0) [label=below:{$-3$}] {};

\draw (04)--(03)--(02)--(01)--(00);
\draw (-10)--(00)--(10);
\end{scope}

\begin{scope}[shift={(15,0)}]
\node[empty] (-12) at (-1,2) [] {$\Gamma_4$};

\node[rectangle] (02) at (0,2) [label=right:{$-4$}] {};
\node[bullet] (01) at (0,1) [label=right:{$-1$}] {};
\node[rectangle] (-10) at (-1,0) [label=below:{$-2$}] {};
\node[rectangle] (00) at (0,0) [label=below:{$-5$}] {};
\node[bullet] (10) at (1,0) [label=below:{$-1$}] {};
\node[rectangle] (20) at (2,0) [label=below:{$-4$}] {};

\draw (02)--(01)--(00);
\draw (-10)--(00)--(10)--(20);
\end{scope}

\begin{scope}[shift={(0,-4)}]
\node[empty] (-12) at (-1,2) [] {$\Gamma_5$};

\node[rectangle] (02) at (0,2) [label=right:{$-4$}] {};
\node[bullet] (01) at (0,1) [label=right:{$-1$}] {};
\node[rectangle] (-10) at (-1,0) [label=below:{$-2$}] {};
\node[rectangle] (00) at (0,0) [label=below:{$-5$}] {};
\node[bullet] (10) at (1,0) [label=below:{$-1$}] {};
\node[rectangle] (20) at (2,0) [label=below:{$-4$}] {};
\node[bullet] (30) at (3,0) [label=below:{$-1$}] {};
\node[rectangle] (40) at (4,0) [label=below:{$-4$}] {};

\draw (02)--(01)--(00);
\draw (-10)--(00)--(10)--(20)--(30)--(40);
\end{scope}

\begin{scope}[shift={(8,-4)}]
\node[empty] (-12) at (-1,2) [] {$\Gamma_6$};

\node[rectangle] (02) at (0,2) [label=right:{$-4$}] {};
\node[bullet] (01) at (0,1) [label=right:{$-1$}] {};
\node[rectangle] (-10) at (-1,0) [label=below:{$-2$}] {};
\node[rectangle] (00) at (0,0) [label=below:{$-5$}] {};
\node[rectangle] (10) at (1,0) [label=below:{$-3$}] {};
\node[bullet] (20) at (2,0) [label=below:{$-1$}] {};
\node[rectangle] (30) at (3,0) [label=below:{$-4$}] {};

\draw (02)--(01)--(00);
\draw (-10)--(00)--(10)--(20)--(30);
\end{scope}

\begin{scope}[shift={(16,-4)}]
\node[empty] (-22) at (-2,2) [] {$\Gamma_7$};

\node[rectangle] (01) at (0,1) [label=right:{$-2$}] {};
\node[rectangle] (-20) at (-2,0) [label=below:{$-4$}] {};
\node[bullet] (-10) at (-1,0) [label=below:{$-1$}] {};
\node[rectangle] (00) at (0,0) [label=below:{$-6$}] {};
\node[rectangle] (10) at (1,0) [label=below:{$-2$}] {};
\node[rectangle] (20) at (2,0) [label=below:{$-3$}] {};

\draw (01)--(00);
\draw (-20)--(-10)--(00)--(10)--(20);
\end{scope}

\begin{scope}[shift={(0,-8)}]
\node[empty] (-11) at (-1,1) [] {$\Gamma_8$};

\node[rectangle] (01) at (0,1) [label=right:{$-2$}] {};
\node[rectangle] (-10) at (-1,0) [label=below:{$-3$}] {};
\node[rectangle] (00) at (0,0) [label=below:{$-5$}] {};
\node[bullet] (10) at (1,0) [label=below:{$-1$}] {};
\node[rectangle] (20) at (2,0) [label=below:{$-2$}] {};
\node[rectangle] (30) at (3,0) [label=below:{$-2$}] {};
\node[rectangle] (40) at (4,0) [label=below:{$-6$}] {};

\draw (01)--(00);
\draw (-10)--(00)--(10)--(20)--(30)--(40);
\end{scope}

\begin{scope}[shift={(12,-8)}]
\node[empty] (-31) at (-3,1) [] {$\Gamma_9$};

\node[rectangle] (01) at (0,1) [label=right:{$-2$}] {};
\node[rectangle] (-30) at (-3,0) [label=below:{$-5$}] {};
\node[rectangle] (-20) at (-2,0) [label=below:{$-2$}] {};
\node[bullet] (-10) at (-1,0) [label=below:{$-1$}] {};
\node[rectangle] (00) at (0,0) [label=below:{$-5$}] {};
\node[bullet] (10) at (1,0) [label=below:{$-1$}] {};
\node[rectangle] (20) at (2,0) [label=below:{$-4$}] {};

\draw (01)--(00);
\draw (-30)--(-20)--(-10)--(00)--(10)--(20);
\end{scope}
\end{tikzpicture}

\caption{Non-linear maximal subgraph $\Gamma$}
\label{figure:subgraph-Gamma}
\end{figure}


\begin{figure}
\begin{tikzpicture}[scale=0.5]
\begin{scope}[shift={(0,0)}]
\node[bullet] (01) at (0,1) [label=right:{$-3$}] {};
\node[bullet] (-10) at (-1,0) [label=below:{$-2$}] {};
\node[bullet] (00) at (0,0) [label=below:{$-4$}] {};
\node[bullet] (10) at (1,0) [label=below:{$-3$}] {};

\draw (01)--(00);
\draw (-10)--(00)--(10);
\end{scope}

\draw[->,decorate,decoration={snake,amplitude=.4mm,segment length=2mm,post length=1mm}] (1.5,1)--(3.5,1) node [above,align=center,midway]
{antiflip};

\begin{scope}[shift={(5,0)}]
\node[rectangle] (02) at (0,2) [label=right:{$-4$}] {};
\node[bullet] (01) at (0,1) [label=right:{$-1$}] {};
\node[bullet] (-10) at (-1,0) [label=below:{$-2$}] {};
\node[bullet] (00) at (0,0) [label=below:{$-5$}] {};
\node[bullet] (10) at (1,0) [label=below:{$-3$}] {};

\draw (02)--(01)--(00);
\draw (-10)--(00)--(10);
\end{scope}

\draw[->,decorate,decoration={snake,amplitude=.4mm,segment length=2mm,post length=1mm}] (6.5,1)--(8.5,1) node [above,align=center,midway]
{$\mathbb{Q}$-BLDN};

\begin{scope}[shift={(10,0)}]
\node[rectangle] (02) at (0,2) [label=right:{$-4$}] {};
\node[bullet] (01) at (0,1) [label=right:{$-1$}] {};
\node[rectangle] (-10) at (-1,0) [label=below:{$-2$}] {};
\node[rectangle] (00) at (0,0) [label=below:{$-5$}] {};
\node[rectangle] (10) at (1,0) [label=below:{$-3$}] {};

\draw (02)--(01)--(00);
\draw (-10)--(00)--(10);
\end{scope}
\end{tikzpicture}
\caption{A sequence of rational blow-downs for $\Gamma_1$}
\label{figure:sequence-Gamma_1}
\end{figure}
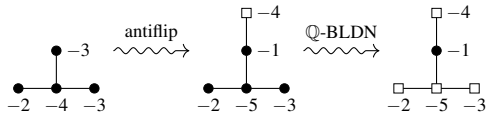


\begin{figure}
\begin{tikzpicture}[scale=0.5]
\begin{scope}[shift={(0,0)}]

\node[bullet] (01) at (0,1) [label=right:{$-4$}] {};
\node[bullet] (-10) at (-1,0) [label=below:{$-2$}] {};
\node[bullet] (00) at (0,0) [label=below:{$-3$}] {};
\node[bullet] (10) at (1,0) [label=below:{$-3$}] {};

\draw (01)--(00);
\draw (-10)--(00)--(10);
\end{scope}

\draw[->,decorate,decoration={snake,amplitude=.4mm,segment length=2mm,post length=1mm}] (1.5,1)--(3.5,1) node [above,align=center,midway]
{antiflip};

\begin{scope}[shift={(5,0)}]

\node[rectangle] (03) at (0,3) [label=right:{$-5$}] {};
\node[rectangle] (02) at (0,2) [label=right:{$-2$}] {};
\node[bullet] (01) at (0,1) [label=right:{$-1$}] {};
\node[bullet] (-10) at (-1,0) [label=below:{$-2$}] {};
\node[bullet] (00) at (0,0) [label=below:{$-5$}] {};
\node[bullet] (10) at (1,0) [label=below:{$-3$}] {};

\draw (03)--(02)--(01)--(00);
\draw (-10)--(00)--(10);
\end{scope}

\draw[->,decorate,decoration={snake,amplitude=.4mm,segment length=2mm,post length=1mm}] (6.5,1)--(8.5,1) node [above,align=center,midway]
{$\mathbb{Q}$-BLDN};

\begin{scope}[shift={(10,0)}]
\node[bullet] (02) at (0,2) [label=right:{$-3$}] {};
\node[bullet] (01) at (0,1) [label=right:{$-2$}] {};
\node[bullet] (-10) at (-1,0) [label=below:{$-2$}] {};
\node[bullet] (00) at (0,0) [label=below:{$-3$}] {};
\node[bullet] (10) at (1,0) [label=below:{$-3$}] {};

\draw (02)--(01)--(00);
\draw (-10)--(00)--(10);
\end{scope}
\end{tikzpicture}
\caption{A sequence of rational blow-downs for $\Gamma_2$}
\label{figure:sequence-Gamma_2}
\end{figure}
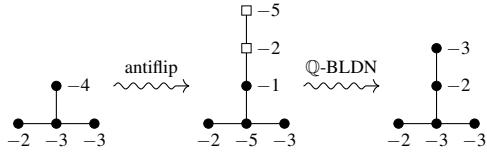


\begin{figure}
\begin{tikzpicture}[scale=0.5]
\begin{scope}[shift={(0,0)}]
\node[bullet] (02) at (0,2) [label=right:{$-3$}] {};
\node[bullet] (01) at (0,1) [label=right:{$-2$}] {};
\node[bullet] (-10) at (-1,0) [label=below:{$-2$}] {};
\node[bullet] (00) at (0,0) [label=below:{$-3$}] {};
\node[bullet] (10) at (1,0) [label=below:{$-3$}] {};

\draw (02)--(01)--(00);
\draw (-10)--(00)--(10);
\end{scope}

\draw[->,decorate,decoration={snake,amplitude=.4mm,segment length=2mm,post length=1mm}] (1.5,1)--(3.5,1) node [above,align=center,midway]
{antiflip};

\begin{scope}[shift={(5,0)}]
\node[rectangle] (03) at (0,3) [label=right:{$-4$}] {};
\node[bullet] (02) at (0,2) [label=right:{$-1$}] {};
\node[bullet] (01) at (0,1) [label=right:{$-3$}] {};
\node[bullet] (-10) at (-1,0) [label=below:{$-2$}] {};
\node[bullet] (00) at (0,0) [label=below:{$-4$}] {};
\node[bullet] (10) at (1,0) [label=below:{$-3$}] {};

\draw (03)--(02)--(01)--(00);
\draw (-10)--(00)--(10);
\end{scope}

\draw[->,decorate,decoration={snake,amplitude=.4mm,segment length=2mm,post length=1mm}] (6.5,1)--(8.5,1) node [above,align=center,midway]
{antiflip};

\begin{scope}[shift={(10,0)}]
\node[rectangle] (04) at (0,4) [label=right:{$-4$}] {};
\node[bullet] (03) at (0,3) [label=right:{$-1$}] {};
\node[rectangle] (02) at (0,2) [label=right:{$-4$}] {};
\node[bullet] (01) at (0,1) [label=right:{$-1$}] {};
\node[bullet] (-10) at (-1,0) [label=below:{$-2$}] {};
\node[bullet] (00) at (0,0) [label=below:{$-5$}] {};
\node[bullet] (10) at (1,0) [label=below:{$-3$}] {};

\draw (04)--(03)--(02)--(01)--(00);
\draw (-10)--(00)--(10);
\end{scope}

\draw[->,decorate,decoration={snake,amplitude=.4mm,segment length=2mm,post length=1mm}] (11.5,1)--(13.5,1) node [above,align=center,midway]
{$\mathbb{Q}$-BLDN};

\begin{scope}[shift={(15,0)}]
\node[rectangle] (04) at (0,4) [label=right:{$-4$}] {};
\node[bullet] (03) at (0,3) [label=right:{$-1$}] {};
\node[rectangle] (02) at (0,2) [label=right:{$-4$}] {};
\node[bullet] (01) at (0,1) [label=right:{$-1$}] {};
\node[rectangle] (-10) at (-1,0) [label=below:{$-2$}] {};
\node[rectangle] (00) at (0,0) [label=below:{$-5$}] {};
\node[rectangle] (10) at (1,0) [label=below:{$-3$}] {};

\draw (04)--(03)--(02)--(01)--(00);
\draw (-10)--(00)--(10);
\end{scope}
\end{tikzpicture}
\caption{A sequence of rational blow-downs for $\Gamma_3$}
\label{figure:sequence-Gamma_3}
\end{figure}
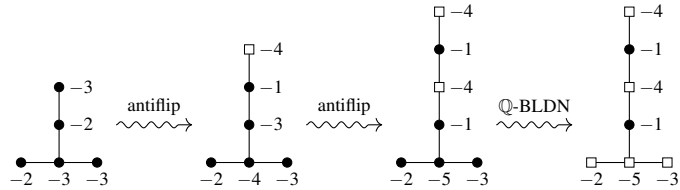


\begin{figure}
\begin{tikzpicture}[scale=0.5]
\begin{scope}[shift={(0,0)}]
\node[bullet] (01) at (0,1) [label=right:{$-3$}] {};
\node[bullet] (-10) at (-1,0) [label=below:{$-2$}] {};
\node[bullet] (00) at (0,0) [label=below:{$-3$}] {};
\node[bullet] (10) at (1,0) [label=below:{$-3$}] {};

\draw (01)--(00);
\draw (-10)--(00)--(10);
\end{scope}

\draw[->,decorate,decoration={snake,amplitude=.4mm,segment length=2mm,post length=1mm}] (1.5,1)--(3.5,1) node [above,align=center,midway]
{antiflip};

\begin{scope}[shift={(5,0)}]
\node[bullet] (01) at (0,1) [label=right:{$-3$}] {};
\node[bullet] (-10) at (-1,0) [label=below:{$-2$}] {};
\node[bullet] (00) at (0,0) [label=below:{$-4$}] {};
\node[bullet] (10) at (1,0) [label=below:{$-1$}] {};
\node[rectangle] (20) at (2,0) [label=below:{$-4$}] {};

\draw (01)--(00);
\draw (-10)--(00)--(10)--(20);
\end{scope}

\draw[->,decorate,decoration={snake,amplitude=.4mm,segment length=2mm,post length=1mm}] (7.5,1)--(9.5,1) node [above,align=center,midway]
{antiflip};

\begin{scope}[shift={(11,0)}]
\node[rectangle] (02) at (0,2) [label=right:{$-4$}] {};
\node[bullet] (01) at (0,1) [label=right:{$-1$}] {};
\node[bullet] (-10) at (-1,0) [label=below:{$-2$}] {};
\node[bullet] (00) at (0,0) [label=below:{$-5$}] {};
\node[bullet] (10) at (1,0) [label=below:{$-1$}] {};
\node[rectangle] (20) at (2,0) [label=below:{$-4$}] {};

\draw (02)--(01)--(00);
\draw (-10)--(00)--(10)--(20);
\end{scope}

\draw[->,decorate,decoration={snake,amplitude=.4mm,segment length=2mm,post length=1mm}] (13.5,1)--(15.5,1) node [above,align=center,midway]
{$\mathbb{Q}$-BLDN};

\begin{scope}[shift={(17,0)}]
\node[rectangle] (02) at (0,2) [label=right:{$-4$}] {};
\node[bullet] (01) at (0,1) [label=right:{$-1$}] {};
\node[rectangle] (-10) at (-1,0) [label=below:{$-2$}] {};
\node[rectangle] (00) at (0,0) [label=below:{$-5$}] {};
\node[bullet] (10) at (1,0) [label=below:{$-1$}] {};
\node[rectangle] (20) at (2,0) [label=below:{$-4$}] {};

\draw (02)--(01)--(00);
\draw (-10)--(00)--(10)--(20);
\end{scope}
\end{tikzpicture}
\caption{A sequence of rational blow-downs for $\Gamma_4$}
\label{figure:sequence-Gamma_4}
\end{figure}
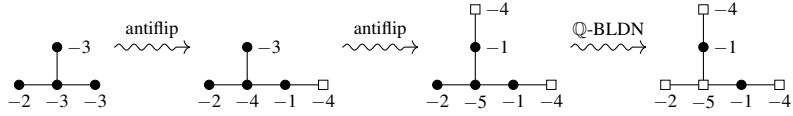


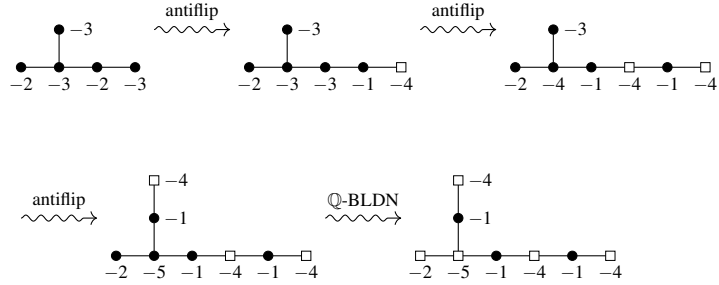
\begin{figure}
\begin{tikzpicture}[scale=0.5]
\begin{scope}[shift={(0,0)}]
\node[bullet] (01) at (0,1) [label=right:{$-3$}] {};
\node[bullet] (-10) at (-1,0) [label=below:{$-2$}] {};
\node[bullet] (00) at (0,0) [label=below:{$-3$}] {};
\node[bullet] (10) at (1,0) [label=below:{$-2$}] {};
\node[bullet] (20) at (2,0) [label=below:{$-3$}] {};

\draw (01)--(00);
\draw (-10)--(00)--(10)--(20);
\end{scope}

\draw[->,decorate,decoration={snake,amplitude=.4mm,segment length=2mm,post length=1mm}] (2.5,1)--(4.5,1) node [above,align=center,midway]
{antiflip};

\begin{scope}[shift={(6,0)}]
\node[bullet] (01) at (0,1) [label=right:{$-3$}] {};
\node[bullet] (-10) at (-1,0) [label=below:{$-2$}] {};
\node[bullet] (00) at (0,0) [label=below:{$-3$}] {};
\node[bullet] (10) at (1,0) [label=below:{$-3$}] {};
\node[bullet] (20) at (2,0) [label=below:{$-1$}] {};
\node[rectangle] (30) at (3,0) [label=below:{$-4$}] {};

\draw (01)--(00);
\draw (-10)--(00)--(10)--(20)--(30);
\end{scope}

\draw[->,decorate,decoration={snake,amplitude=.4mm,segment length=2mm,post length=1mm}] (9.5,1)--(11.5,1) node [above,align=center,midway]
{antiflip};

\begin{scope}[shift={(13,0)}]
\node[bullet] (01) at (0,1) [label=right:{$-3$}] {};
\node[bullet] (-10) at (-1,0) [label=below:{$-2$}] {};
\node[bullet] (00) at (0,0) [label=below:{$-4$}] {};
\node[bullet] (10) at (1,0) [label=below:{$-1$}] {};
\node[rectangle] (20) at (2,0) [label=below:{$-4$}] {};
\node[bullet] (30) at (3,0) [label=below:{$-1$}] {};
\node[rectangle] (40) at (4,0) [label=below:{$-4$}] {};

\draw (01)--(00);
\draw (-10)--(00)--(10)--(20)--(30)--(40);
\end{scope}

\draw[->,decorate,decoration={snake,amplitude=.4mm,segment length=2mm,post length=1mm}] (-1,-4)--(1,-4) node [above,align=center,midway]
{antiflip};

\begin{scope}[shift={(2.5,-5)}]
\node[rectangle] (02) at (0,2) [label=right:{$-4$}] {};
\node[bullet] (01) at (0,1) [label=right:{$-1$}] {};
\node[bullet] (-10) at (-1,0) [label=below:{$-2$}] {};
\node[bullet] (00) at (0,0) [label=below:{$-5$}] {};
\node[bullet] (10) at (1,0) [label=below:{$-1$}] {};
\node[rectangle] (20) at (2,0) [label=below:{$-4$}] {};
\node[bullet] (30) at (3,0) [label=below:{$-1$}] {};
\node[rectangle] (40) at (4,0) [label=below:{$-4$}] {};

\draw (02)--(01)--(00);
\draw (-10)--(00)--(10)--(20)--(30)--(40);
\end{scope}

\draw[->,decorate,decoration={snake,amplitude=.4mm,segment length=2mm,post length=1mm}] (7,-4)--(9,-4) node [above,align=center,midway]
{$\mathbb{Q}$-BLDN};

\begin{scope}[shift={(10.5,-5)}]
\node[rectangle] (02) at (0,2) [label=right:{$-4$}] {};
\node[bullet] (01) at (0,1) [label=right:{$-1$}] {};
\node[rectangle] (-10) at (-1,0) [label=below:{$-2$}] {};
\node[rectangle] (00) at (0,0) [label=below:{$-5$}] {};
\node[bullet] (10) at (1,0) [label=below:{$-1$}] {};
\node[rectangle] (20) at (2,0) [label=below:{$-4$}] {};
\node[bullet] (30) at (3,0) [label=below:{$-1$}] {};
\node[rectangle] (40) at (4,0) [label=below:{$-4$}] {};

\draw (02)--(01)--(00);
\draw (-10)--(00)--(10)--(20)--(30)--(40);
\end{scope}
\end{tikzpicture}
\caption{A sequence of rational blow-downs for $\Gamma_5$}
\label{figure:sequence-Gamma_5}
\end{figure}


\begin{figure}
\begin{tikzpicture}[scale=0.5]
\begin{scope}[shift={(0,0)}]
\node[bullet] (01) at (0,1) [label=right:{$-3$}] {};
\node[bullet] (-10) at (-1,0) [label=below:{$-2$}] {};
\node[bullet] (00) at (0,0) [label=below:{$-4$}] {};
\node[bullet] (10) at (1,0) [label=below:{$-2$}] {};
\node[bullet] (20) at (2,0) [label=below:{$-3$}] {};

\draw (01)--(00);
\draw (-10)--(00)--(10)--(20);
\end{scope}

\draw[->,decorate,decoration={snake,amplitude=.4mm,segment length=2mm,post length=1mm}] (1.5,1)--(3.5,1) node [above,align=center,midway]
{antiflip};

\begin{scope}[shift={(5,0)}]
\node[bullet] (01) at (0,1) [label=right:{$-3$}] {};
\node[bullet] (-10) at (-1,0) [label=below:{$-2$}] {};
\node[bullet] (00) at (0,0) [label=below:{$-4$}] {};
\node[bullet] (10) at (1,0) [label=below:{$-3$}] {};
\node[bullet] (20) at (2,0) [label=below:{$-1$}] {};
\node[rectangle] (30) at (3,0) [label=below:{$-4$}] {};

\draw (01)--(00);
\draw (-10)--(00)--(10)--(20)--(30);
\end{scope}

\draw[->,decorate,decoration={snake,amplitude=.4mm,segment length=2mm,post length=1mm}] (7.5,1)--(9.5,1) node [above,align=center,midway]
{antiflip};

\begin{scope}[shift={(11,0)}]
\node[rectangle] (02) at (0,2) [label=right:{$-4$}] {};
\node[bullet] (01) at (0,1) [label=right:{$-1$}] {};
\node[bullet] (-10) at (-1,0) [label=below:{$-2$}] {};
\node[bullet] (00) at (0,0) [label=below:{$-5$}] {};
\node[bullet] (10) at (1,0) [label=below:{$-3$}] {};
\node[bullet] (20) at (2,0) [label=below:{$-1$}] {};
\node[rectangle] (30) at (3,0) [label=below:{$-4$}] {};

\draw (02)--(01)--(00);
\draw (-10)--(00)--(10)--(20)--(30);
\end{scope}

\draw[->,decorate,decoration={snake,amplitude=.4mm,segment length=2mm,post length=1mm}] (13.5,1)--(15.5,1) node [above,align=center,midway]
{$\mathbb{Q}$-BLDN};

\begin{scope}[shift={(17,0)}]
\node[rectangle] (02) at (0,2) [label=right:{$-4$}] {};
\node[bullet] (01) at (0,1) [label=right:{$-1$}] {};
\node[rectangle] (-10) at (-1,0) [label=below:{$-2$}] {};
\node[rectangle] (00) at (0,0) [label=below:{$-5$}] {};
\node[rectangle] (10) at (1,0) [label=below:{$-3$}] {};
\node[bullet] (20) at (2,0) [label=below:{$-1$}] {};
\node[rectangle] (30) at (3,0) [label=below:{$-4$}] {};

\draw (02)--(01)--(00);
\draw (-10)--(00)--(10)--(20)--(30);
\end{scope}
\end{tikzpicture}
\caption{A sequence of rational blow-downs for $\Gamma_6$}
\label{figure:sequence-Gamma_6}
\end{figure}
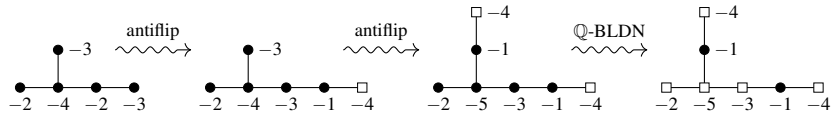


\begin{figure}
\begin{tikzpicture}[scale=0.5]
\begin{scope}[shift={(0,0)}]
\node[bullet] (01) at (0,1) [label=right:{$-2$}] {};
\node[bullet] (-10) at (-1,0) [label=below:{$-3$}] {};
\node[bullet] (00) at (0,0) [label=below:{$-5$}] {};
\node[bullet] (10) at (1,0) [label=below:{$-2$}] {};
\node[bullet] (20) at (2,0) [label=below:{$-3$}] {};

\draw (01)--(00);
\draw (-10)--(00)--(10)--(20);
\end{scope}

\draw[->,decorate,decoration={snake,amplitude=.4mm,segment length=2mm,post length=1mm}] (2.5,1)--(4.5,1) node [above,align=center,midway]
{antiflip};

\begin{scope}[shift={(7,0)}]
\node[bullet] (01) at (0,1) [label=right:{$-2$}] {};
\node[rectangle] (-20) at (-2,0) [label=below:{$-4$}] {};
\node[bullet] (-10) at (-1,0) [label=below:{$-1$}] {};
\node[bullet] (00) at (0,0) [label=below:{$-6$}] {};
\node[bullet] (10) at (1,0) [label=below:{$-2$}] {};
\node[bullet] (20) at (2,0) [label=below:{$-3$}] {};

\draw (01)--(00);
\draw (-20)--(-10)--(00)--(10)--(20);
\end{scope}

\draw[->,decorate,decoration={snake,amplitude=.4mm,segment length=2mm,post length=1mm}] (9.5,1)--(11.1,1) node [above,align=center,midway]
{$\mathbb{Q}$-BLDN};

\begin{scope}[shift={(13.5,0)}]
\node[rectangle] (01) at (0,1) [label=right:{$-2$}] {};
\node[rectangle] (-20) at (-2,0) [label=below:{$-4$}] {};
\node[bullet] (-10) at (-1,0) [label=below:{$-1$}] {};
\node[rectangle] (00) at (0,0) [label=below:{$-6$}] {};
\node[rectangle] (10) at (1,0) [label=below:{$-2$}] {};
\node[rectangle] (20) at (2,0) [label=below:{$-3$}] {};

\draw (01)--(00);
\draw (-20)--(-10)--(00)--(10)--(20);
\end{scope}
\end{tikzpicture}
\caption{A sequence of rational blow-downs for $\Gamma_7$}
\label{figure:sequence-Gamma_7}
\end{figure}
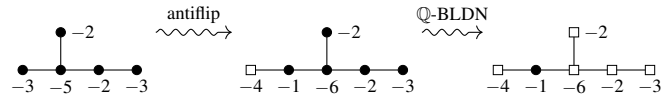


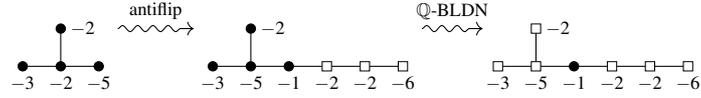
\begin{figure}
\begin{tikzpicture}[scale=0.5]
\begin{scope}[shift={(0,0)}]
\node[bullet] (01) at (0,1) [label=right:{$-2$}] {};
\node[bullet] (-10) at (-1,0) [label=below:{$-3$}] {};
\node[bullet] (00) at (0,0) [label=below:{$-2$}] {};
\node[bullet] (10) at (1,0) [label=below:{$-5$}] {};

\draw (01)--(00);
\draw (-10)--(00)--(10);
\end{scope}

\draw[->,decorate,decoration={snake,amplitude=.4mm,segment length=2mm,post length=1mm}] (1.5,1)--(3.5,1) node [above,align=center,midway]
{antiflip};

\begin{scope}[shift={(5,0)}]
\node[bullet] (01) at (0,1) [label=right:{$-2$}] {};
\node[bullet] (-10) at (-1,0) [label=below:{$-3$}] {};
\node[bullet] (00) at (0,0) [label=below:{$-5$}] {};
\node[bullet] (10) at (1,0) [label=below:{$-1$}] {};
\node[rectangle] (20) at (2,0) [label=below:{$-2$}] {};
\node[rectangle] (30) at (3,0) [label=below:{$-2$}] {};
\node[rectangle] (40) at (4,0) [label=below:{$-6$}] {};

\draw (01)--(00);
\draw (-10)--(00)--(10)--(20)--(30)--(40);
\end{scope}

\draw[->,decorate,decoration={snake,amplitude=.4mm,segment length=2mm,post length=1mm}] (9.5,1)--(11.1,1) node [above,align=center,midway]
{$\mathbb{Q}$-BLDN};

\begin{scope}[shift={(12.5,0)}]
\node[rectangle] (01) at (0,1) [label=right:{$-2$}] {};
\node[rectangle] (-10) at (-1,0) [label=below:{$-3$}] {};
\node[rectangle] (00) at (0,0) [label=below:{$-5$}] {};
\node[bullet] (10) at (1,0) [label=below:{$-1$}] {};
\node[rectangle] (20) at (2,0) [label=below:{$-2$}] {};
\node[rectangle] (30) at (3,0) [label=below:{$-2$}] {};
\node[rectangle] (40) at (4,0) [label=below:{$-6$}] {};

\draw (01)--(00);
\draw (-10)--(00)--(10)--(20)--(30)--(40);
\end{scope}
\end{tikzpicture}
\caption{A sequence of rational blow-downs for $\Gamma_8$}
\label{figure:sequence-Gamma_8}
\end{figure}


\begin{figure}
\begin{tikzpicture}[scale=0.5]
\begin{scope}[shift={(0,0)}]
\node[bullet] (01) at (0,1) [label=right:{$-2$}] {};
\node[bullet] (-10) at (-1,0) [label=below:{$-4$}] {};
\node[bullet] (00) at (0,0) [label=below:{$-2$}] {};
\node[bullet] (10) at (1,0) [label=below:{$-3$}] {};

\draw (01)--(00);
\draw (-10)--(00)--(10);
\end{scope}

\draw[->,decorate,decoration={snake,amplitude=.4mm,segment length=2mm,post length=1mm}] (1.5,1)--(3.5,1) node [above,align=center,midway]
{antiflip};

\begin{scope}[shift={(6,0)}]
\node[bullet] (01) at (0,1) [label=right:{$-2$}] {};
\node[rectangle] (-30) at (-3,0) [label=below:{$-5$}] {};
\node[rectangle] (-20) at (-2,0) [label=below:{$-2$}] {};
\node[bullet] (-10) at (-1,0) [label=below:{$-1$}] {};
\node[bullet] (00) at (0,0) [label=below:{$-4$}] {};
\node[bullet] (10) at (1,0) [label=below:{$-3$}] {};

\draw (01)--(00);
\draw (-30)--(-20)--(-10)--(00)--(10);
\end{scope}

\draw[->,decorate,decoration={snake,amplitude=.4mm,segment length=2mm,post length=1mm}] (7.5,1)--(9.5,1) node [above,align=center,midway]
{antiflip};

\begin{scope}[shift={(12,0)}]
\node[bullet] (01) at (0,1) [label=right:{$-2$}] {};
\node[rectangle] (-30) at (-3,0) [label=below:{$-5$}] {};
\node[rectangle] (-20) at (-2,0) [label=below:{$-2$}] {};
\node[bullet] (-10) at (-1,0) [label=below:{$-1$}] {};
\node[bullet] (00) at (0,0) [label=below:{$-5$}] {};
\node[bullet] (10) at (1,0) [label=below:{$-1$}] {};
\node[rectangle] (20) at (2,0) [label=below:{$-4$}] {};

\draw (01)--(00);
\draw (-30)--(-20)--(-10)--(00)--(10)--(20);
\end{scope}

\draw[->,decorate,decoration={snake,amplitude=.4mm,segment length=2mm,post length=1mm}] (14.5,1)--(16.5,1) node [above,align=center,midway]
{$\mathbb{Q}$-BLDN};

\begin{scope}[shift={(19,0)}]
\node[rectangle] (01) at (0,1) [label=right:{$-2$}] {};
\node[rectangle] (-30) at (-3,0) [label=below:{$-5$}] {};
\node[rectangle] (-20) at (-2,0) [label=below:{$-2$}] {};
\node[bullet] (-10) at (-1,0) [label=below:{$-1$}] {};
\node[rectangle] (00) at (0,0) [label=below:{$-5$}] {};
\node[bullet] (10) at (1,0) [label=below:{$-1$}] {};
\node[rectangle] (20) at (2,0) [label=below:{$-4$}] {};

\draw (01)--(00);
\draw (-30)--(-20)--(-10)--(00)--(10)--(20);
\end{scope}
\end{tikzpicture}
\caption{A sequence of rational blow-downs for $\Gamma_9$}
\label{figure:sequence-Gamma_9}
\end{figure}
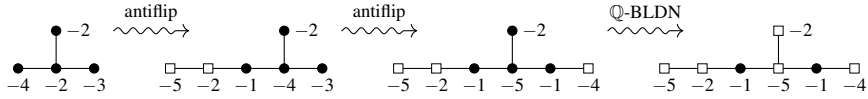

\end{document}